\documentclass[11pt,reqno]{article}
\usepackage{amsmath, amssymb, amsthm, mathrsfs}
\usepackage{amsfonts}
\usepackage{cases}
\usepackage{graphicx}
\usepackage{xcolor}
\usepackage[margin=0.8in]{geometry}
\usepackage[utf8]{inputenc}
\usepackage{verbatim}
\usepackage{caption,subcaption}
\usepackage[colorlinks,citecolor=blue,urlcolor=blue,linkcolor=blue]{hyperref}
\usepackage{bm}
\usepackage{algorithm}
\usepackage{algpseudocode}
\usepackage[normalem]{ulem}
\makeatletter
\renewcommand{\fnum@algorithm}{\fname@algorithm}

\numberwithin{equation}{section}
\newtheorem{Definition}{Definition}[section]
\newtheorem{Remark}{Remark}[section]
\newtheorem{Theorem}{Theorem}[section]
\newtheorem{Lemma}{Lemma}[section]
\newtheorem{Proposition}{Proposition}[section]

\newtheorem{Assumption}{Assumption}[section]

\newcommand{\be}{\begin{equation}}
 \newcommand{\ee}{\end{equation}}
\newcommand{\bee}{\begin{equation*}}
 \newcommand{\eee}{\end{equation*}}
\newcommand{\bi}{\begin{itemize}}
 \newcommand{\ei}{\end{itemize}}
\DeclareMathOperator{\tr}{tr}
\DeclareMathOperator*{\argmax}{arg\,max}

\usepackage{algorithm}
\usepackage{algpseudocode}
\usepackage{soul}

\allowdisplaybreaks

\usepackage{accents}

\def \rightarrow{{\to}}
\def \E{\mathbb{E}}

\def \N{\mathbb{N}}
\def \P{\mathbb{P}}

\def \R{\mathbb{R}}

\def \Cc{{\mathcal C}}

\def \Ac{{\mathcal A}}
\def \Ec{{\mathcal E}}
\def \Pc{{\mathscr P}}

\def \Mc{{\mathcal M}}
\def \Wc{{\mathcal W}}
\def \eps{\varepsilon}
\def \Leb{\operatorname{\texttt{Leb}}}

\def \ul{\operatorname{\textbf{ul}}}
\def \wg{\operatorname{\textbf{wg}}}

\title{Existence of Relaxed Equilibrium for Time-Inconsistent Mean Field Games: A Set-Valued Fixed Point Approach}

\author{
 Zhenhua Wang\thanks{Zhongtai Securities Institute for Financial Studies, Shandong University, Jinan, Shandong, China. Email: \url{zhenhuaw@sdu.edu.cn}}
\and Zhou Zhou\thanks{School of Mathematics and Statistics, University of Sydney, Sydney, Australia. Email: \url{zhou.zhou@sydney.edu.au}}
}

\begin{document}
\date{}
\maketitle

\begin{abstract}	
This paper studies the existence of relaxed equilibria for finite-horizon continuous-time time-inconsistent mean field games. We work directly on the product space of relaxed feedback policies and population flows. The policy component is endowed with the stable topology of Young measures, while the population component is restricted to a compact  convex set of Wasserstein-continuous flows with uniform moment and time-regularity bounds. For every policy-flow pair, we establish uniform Sobolev and H\"older estimates for the associated auxiliary value function and prove its stability under Young-measure convergence of policies and uniform Wasserstein convergence of population flows. We also establish continuity of the induced population-flow map. The latter requires a duality argument for the Fokker--Planck equations because Young-measure convergence yields only weak-$*$ convergence of the controlled drifts. We then construct a set-valued best-response/consistency map with nonempty compact convex values and apply the Kakutani--Fan--Glicksberg fixed-point theorem. The resulting fixed point satisfies both the equilibrium response condition of the intra-personal game and the mean field consistency condition, , thereby establishing the existence of a relaxed equilibrium.\\\\


\end{abstract}

	\section{Introduction}
	
Mean field game (MFG) theory, pioneered independently by Lasry and Lions \cite{lasry2007mean} and Huang, Malham\'e, and Caines \cite{huang2006large}, provides a tractable framework for strategic interactions in large populations of agents. In a standard MFG, a representative agent optimizes an individual objective for a prescribed population distribution, while equilibrium requires consistency between the prescribed distribution and the law generated by the optimal response. A classical analytical formulation describes this equilibrium through a coupled Hamilton--Jacobi--Bellman (HJB) and Fokker--Planck (FP) system. Probabilistic formulations based on forward--backward stochastic differential equations provide another important approach; see, e.g., \cite{carmona2018probabilistic, carmona_probabilistic_2013} and the references therein.

Relaxation and compactification methods form another important part of the existence theory for MFGs. By enlarging strict controls to relaxed controls, or by working with laws of controlled processes or occupation measures, one gains compactness and convexity that may be unavailable at the level of pointwise controls. Such ideas have been extensively used in time-consistent stochastic control and MFGs; see, e.g., \cite{carmona2018probabilistic, common_2016, guo2022entropy, lacker2015mean} and the references therein. They are especially useful when a Hamiltonian optimizer is nonunique or lacks the regularity required to define a continuous strict feedback response.

The present paper concerns MFGs in which the representative agent's objective is \emph{time-inconsistent}. Time inconsistency arises naturally in economic and financial models: A policy preferred at the initial decision time may cease to be preferred by a future self. So the usual dynamic programming principle and the associated notion of a globally optimal policy are no longer appropriate. Following the seminal idea of Strotz \cite{Strotz}, a widely adopted approach is to view the decision maker as a collection of selves indexed by time and to seek an equilibrium of the resulting intra-personal game. This equilibrium approach has led to a substantial literature on time-inconsistent stochastic control; see, among others, \cite{Bjork17,bjork2021time, he2021equilibrium, huang2021strong,yong2012time}.

In continuous time, these equilibria are typically characterized by an extended HJB system or an equilibrium HJB (EHJB) equation, which is nonlinear and nonlocal. Classical verification arguments require substantial solution regularity, yet classical solvability is difficult to obtain under general assumptions. This regularity issue is one of the main obstacles to general equilibrium existence results for time-inconsistent diffusion control. A closely related line of work uses entropy regularization to overcome some of these difficulties. In particular, \cite{bayraktar2025relaxed} studies relaxed equilibria for time-inconsistent Markov decision processes and develops an entropy-regularization approach to equilibrium existence. This work demonstrates how relaxation and compactness can restore existence when strict equilibria need not exist. For continuous-time time-inconsistent diffusion control, \cite{wang2026equilibrium} develops a vanishing-entropy approach under which regularized EHJB solutions converge to a strong solution of the original EHJB, leading to equilibrium existence without requiring a classical solution. Extending this program to MFGs, \cite{bayraktarWangyuzhang2026equilibrium} establishes existence of relaxed equilibria for general continuous-time time-inconsistent MFGs through entropy regularization and a vanishing-entropy limit, together with a detailed analysis of the population flows.

There are also recent existence results that do not rely on entropy regularization. 
 In discrete time, \cite{zhou2026existence} establishes existence of Markov relaxed equilibria for general time-inconsistent stochastic games with uncountable state spaces, including mean field games, by exploiting the compactness of strategies under a weak-star type topology. 
 The work \cite{BW2025} studies an infinite-horizon discrete-time time-inconsistent MFG, introduces equilibrium notions for sophisticated agents, and relates them to the corresponding finite-player games. These results provide discrete-time counterparts to the continuous-time problem considered here.

Compared with the extensive literature on time-consistent MFGs, the theory of time-inconsistent MFGs is still relatively limited. A number of works study linear--quadratic or other structured models in which equilibria can be characterized more explicitly; see, e.g., \cite{guan2022time,moon2020linear,ni2018}. 
The distinction is important: in a time-inconsistent MFG, one must simultaneously resolve the intra-personal equilibrium problem of the representative agent and the consistency of the population distribution.

In this paper, we establish the existence of relaxed equilibria for a general class of finite-horizon continuous-time time-inconsistent MFGs \emph{directly, without introducing entropy regularization}. For a relaxed feedback policy $\pi$ and a candidate population flow $m$, let $V^{\pi,m}$ denote the associated auxiliary value function and $\mu^{\pi,m}$ the population flow induced by the controlled state process. We formulate equilibrium as a fixed point of a set-valued correspondence whose policy component consists of relaxed maximizers of the equilibrium Hamiltonian determined by $D_xV^{\pi,m}(t,t,x)$, while its population component is $\mu^{\pi,m}$.

The main analytical issue is  the continuity of this correspondence under a topology weak enough to make the policy space compact. We identify relaxed feedback policies with Young measures and use their stable topology. Uniform parabolic Sobolev and H\"older estimates yield stability of $V^{\pi,m}$ under simultaneous Young-measure convergence of policies and uniform Wasserstein convergence of population flows. For the population component, Young-measure convergence gives only weak-$*$ convergence of the controlled drifts, so standard strong SDE stability is unavailable. We instead combine a backward-PDE/Fokker--Planck duality argument with compactness of the state densities to prove the continuity of $(\pi,m)\mapsto\mu^{\pi,m}$.

These ingredients allow us to apply the Kakutani--Fan--Glicksberg fixed-point theorem on a compact convex product space of relaxed policies and population flows. The fixed-point policy satisfies an integrated Hamiltonian maximization condition, which is converted into the pointwise almost-everywhere equilibrium Hamiltonian condition by measurable selection. Finally, an It\^o--Krylov verification argument for the resulting strong EHJB solution shows that the fixed point satisfies the intra-personal equilibrium response condition, while the population component of the fixed-point identity gives mean field consistency.

Our approach is complementary to the entropy-regularization method of \cite{bayraktarWangyuzhang2026equilibrium}. That work establishes regularized equilibria through Gibbs-form responses, obtains an equilibrium of the original MFG by sending the entropy weight to zero, and proposes a policy iteration algorithm whose convergence is proved under a short horizon and weak terminal interaction. Entropy regularization is crucial there: it produces the smooth, single-valued Gibbs response that makes both the fixed-point construction and the policy-iteration analysis tractable. 
Here we work directly with the original unregularized problem: the possible nonuniqueness of Hamiltonian maximizers is handled by a set-valued policy correspondence, and compactness is supplied by the Young-measure topology. Some of the stability mechanisms that arise in the vanishing-entropy analysis---in particular, compactness of auxiliary value functions and the treatment of weakly convergent controlled drifts through Fokker--Planck arguments---remain essential, but they are used here to establish continuity of the fixed-point correspondence itself rather than convergence from a regularized problem. Thus the two approaches provide distinct and complementary existence mechanisms for relaxed equilibria in continuous-time time-inconsistent MFGs.


The rest of the paper is organized as follows. Section \ref{subsect:nota} collects the notation and function spaces used throughout the paper. Section \ref{sec:model} introduces the finite-horizon time-inconsistent MFG, relaxed feedback policies, and the notion of equilibrium. Section \ref{sec:fixedpoint} develops the compact policy-flow space, establishes the uniform estimates and stability results for the auxiliary value functions and population flows, constructs the set-valued best-response/consistency correspondence, and proves the existence of a relaxed equilibrium by the Kakutani--Fan--Glicksberg fixed-point theorem.

\subsection{Notations}\label{subsect:nota} 
Let $\mathbb{N}$ be the set of all positive integers and $T>0$ be the finite horizon. In the $m$-dimensional Euclidean space $\R^m$,  we denote by $| \cdot|$ the Euclidean norm, by $\Leb(\cdot)$ the Lebesgue measure, and by $B_r(x)$ the ball centered at $x$ with radius $r$.

Given a function $w(t,x):[0,T]\times\R^d\rightarrow \R$, let $\partial_tw$ denote the right derivative on the time variable $t$ and $D_xw$ (resp. $D^2_xw$) denote the first (resp. second) partial derivative vector (Hessian) on the space variable $x$. 
Let $D$ be a generic domain in $[0,T] \times \R^d$. We first recall the notations of parabolic H\"older spaces for $\alpha \in (0,1)$:
\begin{align*}
 \|w\|_{\Cc^0(D)} &:= \sup_{(t,x)\in D} |w(t,x)|= \|w\|_{L^\infty(D)},\quad
 [w]_{\Cc_{\alpha}(D)} := \sup_{\substack{(t,x)\neq (s,y)\in D\\ (|t-s|+|x-y|^2)\leq 1}} \frac{|w(t,x)-w(s,y)|}{(|t-s|+|x-y|^2)^{\alpha/2}}, \\
 \|w\|_{\Cc_{\alpha}(D)} &:= [w]_{\Cc_{\alpha}(D)}+\|w\|_{\Cc^0(D)}.
\end{align*}
Given a multi-index $a=(a_1,\cdots, a_d)$ with $|a|_{l_1}:= \sum_{i=1}^d a_i$, define $D_x^a w:= \frac{\partial^{|a|_{l_1}} w }{\partial^{a_1}_{x_1}\cdots \partial^{a_d}_{x_d}}$
and we say $w\in \Cc^{l, k}_{\alpha}(D)$ 
(resp. $w\in \Cc^{l, k}(D)$) 
if 
$$
\|w\|_{\Cc^{l, k}_{\alpha}(D)}:= \sum_{0\leq b\leq l, 0\leq |a|_{l_1}\leq k} \|\partial^b_t D^a_x w \|_{\Cc_{\alpha}(D)}<\infty\; \Big(\text{resp. $\|w\|_{\Cc^{l, k}(D)}:=  \sum_{0\leq b\leq l, 0\leq |a|_{l_1}\leq k}\|\partial^b_t D^k_xw \|_{\Cc^0(D)}<\infty$} \Big).
$$ 
Define $D_N(t_0, x_0):= ((t_0, t_0+N)\cap [0,T])\times B_N(x_0) $ and write $D_N:= D_N(0,0)$, for $N>0$. 
We further define the following weighted  H\"older norm on $[0,T]\times \R^d$:
\begin{equation*}
 \|w\|_{\Cc^{l, k, \wg}_{\alpha}([0,T]\times \R^d)}:=\sum_{N\in \mathbb{N}} \frac{1}{2^N} \|w\|_{\Cc^{l, k}_{\alpha}(D_N)}.
\end{equation*}
Given $w(t,s,x)$ defined on the domain $\Delta_{[0,T]}\times \R^d$ with $\Delta_{[0,T]}:=\{(t,s):0\leq t\leq s\leq T\}$ and a spatial subset $Q \subset \R^d$, we further define the following ``$t$-H\"older norms": 
 \be\label{eq:t_norms}
    \begin{aligned}
 \|w\|_{\Cc^{l, k}_{\alpha, [t_0, t_1]\times Q}} &:= \sup_{t\in[t_0, t_1]} \|w(t,\cdot,\cdot)\|_{\Cc^{l, k}_{\alpha}([t, T]\times Q)},\,
 \|w\|_{\Cc^{l, k}_{[t_0, t_1]\times Q}} := \sup_{t\in[t_0, t_1]} \|w(t,\cdot,\cdot)\|_{\Cc^{l, k}([t, T]\times Q)},\\
 \|w\|_{\widetilde{\Cc}^{l, k}_{\alpha,[t_0, t_1]\times Q}} &:= \sup_{t\in[t_0, t_1]} \|w(t,\cdot,\cdot)\|_{\Cc^{l, k}_{\alpha}([t, T]\times Q)} + \sup_{t\in[t_0, t_1]} \|\partial_t w(t,\cdot,\cdot)\|_{\Cc^{l, k}_{\alpha}([t, T]\times Q)},\\
 \|w\|_{\widetilde{\Cc}^{l, k}_{[t_0, t_1]\times Q}} &:= \sup_{t\in[t_0, t_1]} \|w(t,\cdot,\cdot)\|_{\Cc^{l, k}([t, T]\times Q)} + \sup_{t\in[t_0, t_1]} \|\partial_t w(t,\cdot,\cdot)\|_{\Cc^{l, k}([t, T]\times Q)}.
\end{aligned}
\ee
When $Q=\R^d$, we simply omit the spatial domain in the subscript. 

We recall the Sobolev norm $\|w\|_{W^{1,2}_p(D)} := \sum_{2l+ |a| \leq 2} \|\partial^l_t D^a_x w \|_{L^p(D)}$ for a bounded domain $D$, and define the following ``uniformly local Sobolev norm" on $[t,T]\times \R^d$:
$$
\|w(t,\cdot,\cdot)\|_{W^{1,2,\ul}_p([t,T]\times\R^d)}:=\sup_{(s,x)\in [t, T]\times \R^d} \|w(t,\cdot, \cdot) \|_{W^{1,2}_p([s,s+1]\cap[0,T]\times B_1(x))}.
$$
We further define the ``$t$-Sobolev norm":
 \be\label{eq:t_norms2}
    \begin{aligned}
 \|w\|_{\widetilde{W}^{1,2}_{p,[t_0, t_1]}} &:= \sup_{t\in[t_0, t_1]}\|w(t,\cdot,\cdot)\|_{W^{1,2,\ul}_p([t,T]\times\R^d)},\\
 \|w\|_{\widehat{W}^{1,2}_{p,[t_0, t_1]}} &:= \sup_{t\in[t_0, t_1]}(\|w(t,\cdot,\cdot)\|_{W^{1,2,\ul}_p([t,T]\times\R^d)}+\|\partial_tw(t,\cdot,\cdot)\|_{W^{1,2,\ul}_p([t,T]\times\R^d)}). 
\end{aligned} 
\ee

Throughout the paper, we use the respective subscript of a norm to denote the functional space consisting of all functions for which that norm is finite (e.g., $w \in \widetilde{W}^{1,2}_{p,[t_0, t_1]}$ means $\|w\|_{\widetilde{W}^{1,2}_{p,[t_0, t_1]}} < \infty$). Let $0<\alpha<1$ be an arbitrarily fixed H\"older constant, and we shall choose the power $p$ in any Sobolev norm equal to $\frac{d+2}{1-\alpha}$.

Let $(\Omega, \mathcal{F}, \mathbb{F}, \mathbb{P})$ be a complete filtered probability space  supporting an $m$-dimensional standard Brownian motion $W = \{W_t\}_{t \in [0, T]}$. Let $\mathcal{F}_0$ be a given $\sigma$-algebra which is independent of the Brownian motion. Let $\mathbb{F}=\left\{\mathcal{F}_t: 0 \leq t \leq T\right\}$, where $\mathcal{F}_t:=\sigma\left\{\mathcal{F}_0, W(s): 0 \leq s \leq t\right\} \vee \mathcal{N}_0$, and $\mathcal{N}_0$ is the set of all $\mathbb{P}$-null sets. Denote by $\mathscr{P}\left(\R^d\right)$ the space of probability measures on $\R^d$.

Let $L_{\mathcal{F}}^2(\R^d)$ be the collection of $\R^d$-valued random variables with finite second moment, i.e.,
\begin{equation*}
 L_{\mathcal{F}}^2(\R^d) := \left\{\xi: \Omega \rightarrow \R^d \mid \xi \text{ is } \mathcal{F}\text{-measurable with } \mathbb{E}|\xi|^2<\infty \right\}.
\end{equation*}
$L_{\mathcal{F}}^2(\R^d)$ is equipped with the norm $\|\xi\|_{L^2} := \left(\mathbb{E}|\xi|^2\right)^{\frac{1}{2}}$. For any $\xi \in L^2(\R^d)$, denote by $\operatorname{law}(\xi)$ the distribution of $\xi$.

Let $\mathscr{P}_2(\R^d)$ be the space of probability measures with finite second moments equipped with the Wasserstein-2 metric $\Wc_2(\cdot, \cdot)$, i.e.,
\begin{equation*}
 \Wc_2^2(\rho, \gamma) := \inf_{\pi \in \Pi^{\rho, \gamma}} \int_{\R^d \times \R^d} |x-y|^2 \pi(dx, dy),
\end{equation*}
where
\begin{equation*}
 \Pi^{\rho, \gamma} := \left\{\pi \in \mathscr{P}_2(\R^d \times \R^d) : \pi(dx, \R^d)=\rho(dx), \pi(\R^d, dy)=\gamma(dy)\right\}.
\end{equation*}
It is easy to see that
 \be\label{ineq:w2}
 \Wc_2^2(\operatorname{law}(\xi), \operatorname{law}(\eta)) \leq \|\xi-\eta\|_{L^2}^2.
\ee
Let $\mathscr{P}_2^{\kappa, C}(\R^d)$ be a subset of $\mathscr{P}_2(\R^d)$ defined by
\begin{equation*}
 \mathscr{P}_2^{\kappa, C}(\R^d) := \left\{\rho \in \mathscr{P}_2(\R^d) : \int_{\R^d} |x|^{2+\kappa} \rho(dx) \leq C\right\}.
\end{equation*}
Note that $\mathscr{P}_2^{\kappa, C}(\R^d)$ is a compact subset of $\left(\mathscr{P}_2(\R^d), \Wc_2\right)$ for any $\kappa, C>0$.

Throughout the paper, we suppose that $\mathcal{F}_0$ is large enough such that for any $\rho \in \mathscr{P}_2(\R^d)$, there exists a $\xi \in \mathcal{F}_0$ such that $\operatorname{law}(\xi)=\rho$.

Let $\mathscr{M} := C\left([0, T], \mathscr{P}_2(\R^d)\right)$ be the set of $\mathscr{P}_2(\R^d)$-valued continuous curves on $[0, T]$ equipped with the uniform metric $d$, i.e.,
\begin{equation*}
 d(\mu_1, \mu_2) := \sup_{0 \leq t \leq T} \Wc_2(\mu_1(t), \mu_2(t)).
\end{equation*}
Since $\left(\mathscr{P}_2(\R^d), \Wc_2\right)$ is complete, so is $(\mathscr{M}, d)$. Write $\mathscr{M}_{\nu} := \{\mu \in \mathscr{M} : \mu(0)=\nu\}$. {Through this paper, let $0<\kappa<\infty$ be an arbitrarily fixed constant.} Define a subset of $\mathscr{M}_\nu$ by
\begin{equation*}
 \mathscr{M}_\nu^{\kappa, C} := \left\{\mu \in \mathscr{M}_\nu \;\middle|\; 
   \sup_{0 \leq t \neq s \leq T} \frac{\Wc_2^2(\mu(t), \mu(s))}{|t-s|} \leq C,  \;\text{and } \mu(t) \in \mathscr{P}_2^{\kappa, C}(\R^d) \text{ for any } t \in[0, T] 
\right\}.
\end{equation*}
By the well-known Arzel\`a-Ascoli lemma, $\mathscr{M}_\nu^{\kappa, C}$ is a compact and convex subset of $\mathscr{M}$.

 \section{Model Setup}\label{sec:model}
Let $U$ denote the action space, which is a compact subset of $\R^\ell$, and $\mathscr{P}(U)$ denote the set of all probability measures on $U$.
We focus on the model in which only the drift coefficient of the state process is controlled.
\begin{equation*}
 dX^\pi_s = \left( \int_U b(s, X^\pi_s, m_s, a)\pi(s, X^\pi_s, a)da \right) ds + \sigma(s, X^\pi_s, m_s) dW_s 
\end{equation*}
where $b:[0,T]\times \R^d\times \mathscr{P}_2(\R^d)\times U\rightarrow \R^d$, $\sigma:[0,T]\times \R^d\times \mathscr{P}_2(\R^d)\rightarrow \R^{d\times m}$. The payoff functional under the population flow $m$ for applying $\pi$ is then defined by
\begin{equation*}
 J^{\pi, m}(t,x) := \mathbb{E}_{t,x}\left[\int_t^T \left( \int_U r(l-t, X^{\pi}_l, m_l, a)\pi(l, X^{\pi}_l, a)da \right) dl + F(t, X^{\pi}_T, m_T) \right],
\end{equation*}
where $r:[0,T]\times \R^d\times \mathscr{P}_2(\R^d)\times U\rightarrow \R$ and $F:[0,T]\times\R^d\times \mathscr{P}_2(\R^d)\rightarrow \R$ are the reward functions. We further define the auxiliary function, for $(t,s,x)\in\Delta_{[0,T]}\times \R^d$, by
\begin{equation*}
 V^{\pi, m}(t,s,x) := \mathbb{E}_{s,x}\left[\int_s^T  \left( \int_U r(l-t, X^{\pi}_l, m_l, a)\pi(l, X^{\pi}_l, a)da \right) dl + F(t,X^{\pi}_T, m_T)\right].
\end{equation*} 
Here, $J^{\pi, m}(t,x) = V^{\pi, m}(t,t,x)$ represents the payoff functional for entering the game at $(t,x)$, while $V^{\pi, m}(t,s,x)$ represents the payoff reevaluated at a future time-state pair $(s,x)$. 

For the rest of the paper, we use the notation
\begin{equation*}
 \tilde{f}(t,x,m,\varpi):= \int_U f(t,x,m,a)\varpi(da).
\end{equation*}
for any generic function $f(t, x, m, a): [0,T]\times \R^d \times \mathscr{P}_2(\R^d)\times U\rightarrow \R$ and any generic distribution $\varpi\in \mathscr{P}(U)$.

\begin{Definition}
 The set of admissible policies, denoted by $\Ac$, is defined as the collection of all Borel measurable relaxed feedback policies $\pi: [0,T] \times \R^d \rightarrow \mathscr{P}(U)$. 
\end{Definition}

\begin{Assumption}\label{assume.r}
There exist constants $K_1,K_2,\eta>0$ such that the following conditions hold for all $(t,x,m,a) \in [0,T]\times\mathbb{R}^d\times\mathscr{P}_2(\mathbb{R}^d)\times U$:
\bi
\item[(i)] \textbf{Drift coefficient:}
The function $b:[0,T]\times\R^d\times\Pc_2(\R^d)\times U\to\R^d$
is Borel measurable and uniformly bounded. Moreover, 
$$
|b(t,x,m,a)-b(t,x,\rho,a)|\leq K_2\Wc_2(m,\rho),\quad \forall m,\rho\in\Pc_2(\R^d).
$$
\item[(ii)] \textbf{Diffusion coefficient:}
The function $ \sigma:[0,T]\times\R^d\times\Pc_2(\R^d)\to\R^{d\times m}$
is Borel measurable and uniformly bounded. Moreover,
$$
\begin{cases}
|\sigma(t,x,m)-\sigma(s,x,m)|\leq K_1|t-s|^{\alpha/2},\quad \forall t,s\in[0,T],\\
|\sigma(t,x,m)-\sigma(t,y,\rho)|\leq K_2\big(|x-y|+\Wc_2(m,\rho)\big),\quad \forall	x,y\in\R^d,\quad \forall m,\rho\in\Pc_2(\R^d),\\
 \eta|\xi|^2 \leq \xi \sigma\sigma^T(t,x,m) \xi^T, \quad \forall \xi \in \mathbb{R}^d \setminus \{0\}.
\end{cases}
$$
	
\item[(iii)] \textbf{Running reward:} The function $r:[0,T]\times\R^d\times\Pc_2(\R^d)\times U\to\R$ is Borel measurable in all variables and continuously differentiable in its first variable. Moreover,
$$
\begin{cases}
\sup_{(t,x,m,a)\in [0,T]\times \R^d\times \Pc_2(\R^d)\times U}\big(|r(t,x,m,a)|+|\partial_t r(t,x,m,a)|\big)\leq K_1,\\
|r(t,x,m,a)-r(t,x,\rho,a)|+|\partial_t r(t,x,m,a)-\partial_t r(t,x,\rho,a)|\leq K_2\Wc_2(m,\rho),\quad \forall m,\rho\in\Pc_2(\R^d).
\end{cases}
$$
	
\item[(iv)] \textbf{Terminal reward:} The function $F:[0,T]\times\R^d\times\Pc_2(\R^d)\to\R$ is Borel measurable in all variables, and 
$$
\begin{cases}
\|F(\cdot,\cdot,m)\|_{\Cc^{1,2}_\alpha([0,T]\times\R^d)}\leq K_1,\\
|F(t,x,m)-F(t,x,\rho)|+|\partial_tF(t,x,m)-\partial_tF(t,x,\rho)|\leq K_2\Wc_2(m,\rho),\quad \forall m,\rho\in\Pc_2(\R^d).
\end{cases}
$$
	\ei
\end{Assumption}

\begin{Remark}
    Under the conditions of $b$ and $\sigma$ in Assumption \ref{assume.r}, for any admissible policy $\pi \in \Ac$, any measure flow $m \in \mathscr{M}_\nu$, and any initial condition $(t,x) \in[0,T) \times \R^d$, the controlled SDE
    \begin{equation*}
        dX^\pi_s = \tilde{b}(s, X^\pi_s, m_s, \pi(s, X^\pi_s)) ds + \sigma(s, X^\pi_s, m_s) dW_s, \quad X^\pi_t = x, \quad s \in [t,T],
    \end{equation*}
    admits a unique strong solution. Moreover, the boundedness conditions on $r$ and $F$ ensure that $J^{\pi,m}(t,x)$ is always finite.
\end{Remark}

We give the following definition of relaxed equilibrium for the time-inconsistent MFG problem.
\begin{Definition}\label{def:equi.relaxed}
 A pair $(\pi^*, m^*)$, where $\pi^*\in\Ac$  and $m^*\in\mathscr{M}_{\nu}$, is called a (relaxed) mean-field equilibrium if the following conditions are satisfied:
 \begin{itemize}  
  \item[(a)]{\textit{(Equilibrium response in the intra-personal game)}} For any $(t,x)\in[0,T)\times \R^d$ and $\pi'\in\mathcal{A}$,
   \be\label{eq:def.equipi}
   \limsup_{\epsilon\to 0^+} \frac{J^{\pi'\otimes_{t,\epsilon} \pi^*, m^*}(t,x)- J^{\pi^*, m^*}(t,x)}{\epsilon}\leq 0,
  \ee
  where 
   \be\label{eq:def.pipast}
   \pi'\otimes_{t,\epsilon} \pi^*(s,x)=\begin{cases}
    \pi'(s,x), & (s,x)\in[t,t+\epsilon]\times\R^d,\\
    \pi^*(s,x), & otherwise. 
   \end{cases}
  \ee
  
  \item[(b)]{\textit{(Consistency condition on population aggregation)}} $m^*_t=\operatorname{law}(X^{*}_t)$ for all $t\in[0,T]$, where 
   \be\label{eq:def.equim}
   dX^{*}_t=\tilde{b}(t,X^{*}_t,m^*_t,\pi^*(t,X^{*}_t))dt+\sigma(t,X^{*}_t,m^*_t)dW_t,\quad X^*_0\sim\nu.
  \ee
 \end{itemize} 
\end{Definition}

\section{Existence of Equilibrium by a Set-Valued Fixed Point Argument}\label{sec:fixedpoint}

In this section, we establish the existence of an equilibrium directly in the time-inconsistent MFG. We first introduce the topology for relaxed feedback policies and record the uniform estimates for the value functions and population flows. We then prove the continuity properties required for the set-valued fixed-point argument.

\begin{Assumption}\label{assume.lipsa.U}
For every $(t,x,m)\in [0,T]\times \R^d\times \Pc_2(\R^d)$, the functions $a\mapsto b(t,x,m,a)$ and $a\mapsto r(t,x,m,a)$ are continuous on $U$.
\end{Assumption}

Fix a function $h\in C(\R^d)$ satisfying $h(x)>0$ for all $x\in\R^d$ and $\int_{\R^d}h(x)dx=1$. We equip $\Ac$ with the following stable topology: $\pi^n\to\pi$ if
 \be\label{eq:policy.topology}
\int_0^T\int_{\R^d}h(x)\left(\int_U\phi(t,x,a)\pi^n(t,x,da)\right)dxdt
\to
\int_0^T\int_{\R^d}h(x)\left(\int_U\phi(t,x,a)\pi(t,x,da)\right)dxdt
\ee
for every function $\phi:[0,T]\times\R^d\times U\to\R$ which is measurable in $(t,x)$, continuous in $a$ for a.e. $(t,x)$, and satisfies
$$
\int_0^T\int_{\R^d}h(x)\max_{a\in U}|\phi(t,x,a)|dxdt<\infty.
$$
Equivalently, we identify $\pi\in\Ac$ with the finite Young measure $h(x)dtdx\,\pi(t,x,da)$.

\begin{Lemma}\label{lm:policy.compact}
The policy space $\Ac$, endowed with the topology \eqref{eq:policy.topology}, is a compact convex subset of a locally convex topological vector space.
\end{Lemma}

\begin{proof}
Convexity is immediate. For compactness, identify each $\pi\in\Ac$ with the Young measure
$$
\Lambda^\pi(dt,dx,da):=h(x)dtdx\,\pi(t,x,da)
$$
on $[0,T]\times\R^d\times U$. All these measures have the same first marginal $h(x)dtdx$. Since $U$ is compact, the compactness theorem for Young measures (see, e.g., \cite[Chapter IV]{warga2014optimal}) implies that every sequence $(\pi^n)_{n\in\N}\subset\Ac$ admits a subsequence and a Borel measurable kernel $\pi^\infty:[0,T]\times\R^d\to\Pc(U)$ such that \eqref{eq:policy.topology} holds with $\pi=\pi^\infty$. The topology is generated by the family of continuous linear functionals appearing in \eqref{eq:policy.topology}; therefore $\Ac$ can be embedded in the corresponding locally convex space of signed kernels.
\end{proof}

For the fixed-point theorem, we also record a locally convex realization of the topology on the compact flow set. View a flow $m$ as a curve of finite signed measures and, on the linear span of such curves, consider the locally convex topology generated by the seminorms
$$
p_\varphi(m):=\sup_{t\in[0,T]}\left|\int_{\R^d}\varphi(x)m_t(dx)\right|,\qquad \varphi\in C_b(\R^d).
$$

For $\pi\in\Ac$ and $m\in\mathscr{M}_\nu$, let $\mu^{\pi,m}$ denote the population flow induced when all agents use $\pi$ under the input population flow $m$, i.e.,
$$
\mu^{\pi,m}_t=\operatorname{law}(X^{\pi,m}_t),\qquad
dX^{\pi,m}_t=\tilde b(t,X^{\pi,m}_t,m_t,\pi(t,X^{\pi,m}_t))dt+\sigma(t,X^{\pi,m}_t,m_t)dW_t,\quad X^{\pi,m}_0=\xi\sim\nu.
$$

We next give the uniform estimates, which do not require any regularity of policies in $(t,x)$.

\begin{Proposition}\label{prop:uniform.est}
Let Assumption \ref{assume.r} hold and suppose $\nu\in\Pc_{2+\kappa}(\R^d)$. There exist finite constants $A^*,C^*>0$, depending only on the constants in Assumption \ref{assume.r}, $T$, and the $(2+\kappa)$-moment of $\nu$, such that for every $\pi\in\Ac$ and $m\in\mathscr{M}_\nu^{\kappa,C^*}$,
 \be\label{eq:uniform.Vm}
\|V^{\pi,m}\|_{\widetilde{\Cc}^{0,1}_{\alpha,[0,T]}}
\vee
\|V^{\pi,m}\|_{\widehat W^{1,2}_{p,[0,T]}}
\vee
\|J^{\pi,m}\|_{\Cc^{0,1}_{\alpha}([0,T]\times\R^d)}
\leq A^*,\qquad
\mu^{\pi,m}\in\mathscr{M}_\nu^{\kappa,C^*}.
\ee
Moreover, for each fixed $t\in[0,T]$, $V^{\pi,m}(t,\cdot,\cdot)$ is the unique strong solution of
\begin{align}
\partial_sV^{\pi,m}(t,s,x)&+\frac12\tr\left((\sigma\sigma^T)(s,x,m_s)D_x^2V^{\pi,m}(t,s,x)\right)
+\tilde b(s,x,m_s,\pi(s,x))\cdot D_xV^{\pi,m}(t,s,x)\notag\\
&+\tilde r(s-t,x,m_s,\pi(s,x))=0,\quad (s,x)\in[t,T)\times\R^d,\label{eq:linear.V}\\
V^{\pi,m}(t,T,x)&=F(t,x,m_T).\notag
\end{align}
There also exists a constant $C>0$, independent of $(\pi,m)$, such that
 \be\label{eq:uniform.tlip}
|D_xV^{\pi,m}(t_1,s,x)-D_xV^{\pi,m}(t_2,s,x)|\leq C|t_1-t_2|
\ee
for all $t_1,t_2\in[0,T]$ and $(s,x)\in[0,T]\times\R^d$ with $s\geq t_1\vee t_2$.
\end{Proposition}

\begin{proof}
 Throughout the proof, let $A_0$ be a generic positive constant depending only on $d,\alpha,T$, the constants in Assumption \ref{assume.r}, and the $(2+\kappa)$-moment of $\nu$, but independent of $\pi\in\Ac$ and $m\in\mathscr{M}_{\nu}^{\kappa,C^*}$. The constant $A_0$ may vary from line to line.
 
 \textbf{Step 1. }
 We first verify the regularity of the coefficients before invoking the $W^{1,2}_p$ estimate. Fix $(s_1,x),(s_2,y)\in[0,T]\times\R^d$ such that $|s_1-s_2|+|x-y|^2\leq1$. By Assumption \ref{assume.r}(ii), with the measure argument frozen at $m_{s_1}$, 
 \begin{align*}
  |\sigma(s_1,x,m_{s_1})-\sigma(s_2,y,m_{s_2})|&\leq
  |\sigma(s_1,x,m_{s_1})-\sigma(s_2,y,m_{s_1})| +|\sigma(s_2,y,m_{s_1})-\sigma(s_2,y,m_{s_2})|\\
  &\leq K_1(|s_1-s_2|+|x-y|^2)^{\alpha/2} +K_2\Wc_2(m_{s_1},m_{s_2}).
 \end{align*}
 Since $m\in\mathscr{M}_{\nu}^{\kappa,C^*}$,
 $$
 \Wc_2(m_{s_1},m_{s_2})
 \leq \sqrt{C^*}|s_1-s_2|^{1/2}
 \leq \sqrt{C^*}(|s_1-s_2|+|x-y|^2)^{\alpha/2},
 $$
 where the last inequality follows from $|s_1-s_2|+|x-y|^2\leq1$ and $0<\alpha<1$. Therefore,
  \be\label{eq:refined.sigma.holder}
  \|\sigma(\cdot,\cdot,m_\cdot)\|_{\Cc_\alpha([0,T]\times\R^d)}
  \leq A_0\sqrt{C^*}.
 \ee
 Then by the uniform boundedness of $\sigma$ and $|AA^T-BB^T|\leq(|A|+|B|)|A-B|$, we also obtain
  \be\label{eq:refined.ss.holder}
  \|(\sigma\sigma^T)(\cdot,\cdot,m_\cdot)\|_{\Cc_\alpha([0,T]\times\R^d)}
  \leq A_0\sqrt{C^*}.
 \ee
 In particular, the leading coefficient in \eqref{eq:linear.V} is uniformly elliptic and has a $\Cc_\alpha$ norm bounded independently of $(\pi,m)$.
 
 We emphasize that no regularity of $\pi$ in $(s,x)$ is imposed. Consequently,
 the averaged drift $\tilde b(s,x,m_s,\pi(s,x))$ need not be H\"older continuous. This does not affect the $W^{1,2}_p$ estimate used below: the leading coefficient satisfies \eqref{eq:refined.ss.holder}, while the lower-order drift is merely required to be uniformly bounded. By Assumption \ref{assume.r},
  \be\label{eq:refined.br.bound}
  \|\tilde b(\cdot,\cdot,m_\cdot,\pi(\cdot,\cdot))\|_{L^\infty([0,T]\times\R^d)}
  \vee
  \sup_{t\in[0,T]}
  \|\tilde r(\cdot-t,\cdot,m_\cdot,\pi(\cdot,\cdot))\|_{L^\infty([t,T]\times\R^d)}
  \leq A_0.
 \ee
 
 \textbf{Step 2.} Next we estimate the induced population flow.
 For any $\pi\in\Ac$ and $m\in\mathscr{M}_{\nu}^{\kappa,C^*}$, let $X^{\pi,m}$ be the process defining $\mu^{\pi,m}$. By the uniform boundedness of $b$ and $\sigma$, for any $0\leq s<t\leq T$,
 \begin{align*}
  \E|X^{\pi,m}_t-X^{\pi,m}_s|^2 &\leq 2\E\left|\int_s^t \tilde b(u,X^{\pi,m}_u,m_u,\pi(u,X^{\pi,m}_u))du\right|^2+ 2\E\left|\int_s^t\sigma(u,X^{\pi,m}_u,m_u)dW_u\right|^2\\
  &\leq2\|b\|_\infty^2(t-s)^2+2\|\sigma\|_\infty^2(t-s) \leq A_0|t-s|.
 \end{align*}
 Similarly, the BDG inequality gives
 $$
 \sup_{t\in[0,T]}\E|X^{\pi,m}_t|^{2+\kappa}
 \leq A_0\big(\E|\xi|^{2+\kappa}+1\big).
 $$
 It follows that
  \be\label{eq:refined.mu.est}
  \begin{cases}
   \displaystyle
   \sup_{t\in[0,T]}\int_{\R^d}|x|^{2+\kappa}\mu_t^{\pi,m}(dx)  =\sup_{t\in[0,T]}\E|X_t^{\pi,m}|^{2+\kappa}\leq A_0,\\ 
   \displaystyle
   \Wc_2^2(\mu_t^{\pi,m},\mu_s^{\pi,m})   \leq\E|X_t^{\pi,m}-X_s^{\pi,m}|^2   \leq A_0|t-s|.
  \end{cases}
 \ee
 Increasing $C^*$ if necessary, \eqref{eq:refined.mu.est} yields $\mu^{\pi,m}\in\mathscr{M}_{\nu}^{\kappa,C^*}$ uniformly in $(\pi,m)$.
 
 \textbf{Step 3.} We now provide regularities estimates for $V^{\pi,m}$. Fix $t\in[0,T]$. From the stochastic representation and the boundedness of $r$ and $F$,
  \be\label{eq:refined.Vzero}
  \|V^{\pi,m}(t,\cdot,\cdot)\|_{\Cc^0([t,T]\times\R^d)} \leq T\|r\|_\infty+\|F\|_\infty\leq A_0.
 \ee
 Recall that $p=(d+2)/(1-\alpha)$. By \eqref{eq:refined.ss.holder}, the uniform ellipticity in Assumption \ref{assume.r}(ii), and \eqref{eq:refined.br.bound}, the local $W^{1,2}_p$ estimate with truncation arguments (see, e.g., \cite[Theorem 5.2.10]{krylov2008lectures}) is applicable to \eqref{eq:linear.V}. For any $(s_0,x_0)\in[t,T]\times\R^d$,
 \begin{align}
  \|V^{\pi,m}(t,\cdot,\cdot)\|_{W^{1,2}_p(D_T(s_0,x_0))}
  &\leq A_0\Big(  \|V^{\pi,m}(t,\cdot,\cdot)\|_{L^p(D_{T+1}(s_0,x_0))}
  +\|\tilde r(\cdot-t,\cdot,m_\cdot,\pi(\cdot,\cdot))\|_{L^p(D_{T+1}(s_0,x_0))}\notag\\
&\qquad \qquad  +\|F(t,\cdot,m_T)\|_{W^{2,p}(B_{T+1}(x_0))}  \Big)
  \leq A_0. \label{eq:refined.Vsob}
 \end{align}
 Here and below the constant in the Sobolev estimate is uniform over $(\pi,m)$ because the ellipticity constant, the $\Cc_\alpha$ norm of $(\sigma\sigma^T)(\cdot,\cdot,m_\cdot)$, and the $L^\infty$ norm of the lower-order drift are uniformly bounded by \eqref{eq:refined.ss.holder}--\eqref{eq:refined.br.bound}. By the parabolic Sobolev embedding (see, e.g., \cite[Lemma 3.3 in Chapter II]{ladyzhenskaia1968linear}),
 $$
 \|V^{\pi,m}(t,\cdot,\cdot)\|_{\Cc^{0,1}_\alpha(D_T(s_0,x_0))} \leq A_0
 \|V^{\pi,m}(t,\cdot,\cdot)\|_{W^{1,2}_p(D_T(s_0,x_0))} \leq A_0.
 $$
 Taking the supremum over $(s_0,x_0)\in[t,T]\times\R^d$ yields
  \be\label{eq:refined.Vspatial}
  \|V^{\pi,m}(t,\cdot,\cdot)\|_{\Cc^{0,1}_\alpha([t,T]\times\R^d)}
  \vee  \|V^{\pi,m}(t,\cdot,\cdot)\|_{W^{1,2,\ul}_p([t,T]\times\R^d)}
  \leq A_0.
 \ee
 
 We next control the dependence on the initial preference parameter $t$. By Assumption \ref{assume.r}, differentiation of the stochastic representation with respect to the first argument gives
 \begin{align}
  W^{\pi,m}(t,s,x):=\partial_tV^{\pi,m}(t,s,x)
  =\E_{s,x}\bigg[  -\int_s^T  \partial_t\tilde r(l-t,X_l^{\pi,m},m_l,\pi(l,X_l^{\pi,m}))dl  +\partial_tF(t,X_T^{\pi,m},m_T)  \bigg]. \label{eq:refined.Wrep}
 \end{align}
 For every fixed $t$, $W^{\pi,m}(t,\cdot,\cdot)$ satisfies, in the strong sense,
 $$
 \begin{aligned}
  \partial_sW^{\pi,m}(t,s,x)
  &+\frac12\tr\left((\sigma\sigma^T)(s,x,m_s)D_x^2W^{\pi,m}(t,s,x)\right)
  +\tilde b(s,x,m_s,\pi(s,x))\cdot D_xW^{\pi,m}(t,s,x)\\
  &-\partial_t\tilde r(s-t,x,m_s,\pi(s,x))=0,\qquad
  W^{\pi,m}(t,T,x)=\partial_tF(t,x,m_T).
 \end{aligned}
 $$
 The coefficient estimates established in Step 1 are unchanged, and $\partial_t r,\partial_tF$ satisfy the same uniform bounds required by Assumption \ref{assume.r}. Applying exactly the estimate leading to \eqref{eq:refined.Vspatial} to $W^{\pi,m}$ gives
  \be\label{eq:refined.West}
  \sup_{t\in[0,T]}
  \left(
  \|W^{\pi,m}(t,\cdot,\cdot)\|_{\Cc^{0,1}_\alpha([t,T]\times\R^d)}
  \vee
  \|W^{\pi,m}(t,\cdot,\cdot)\|_{W^{1,2,\ul}_p([t,T]\times\R^d)}  \right)
  \leq A_0.
 \ee
 Combining \eqref{eq:refined.Vspatial} and \eqref{eq:refined.West} gives
 $
 \|V^{\pi,m}\|_{\widetilde{\Cc}^{0,1}_{\alpha,[0,T]}}
 \vee \|V^{\pi,m}\|_{\widehat W^{1,2}_{p,[0,T]}} \leq A_0.
$
 
 Finally, let $J^{\pi,m}(t,x)=V^{\pi,m}(t,t,x)$. For $(s,x),(t,y)\in[0,T]\times\R^d$ with $|s-t|+|x-y|^2\leq1$, assume without loss of generality that $t\leq s$. Then
 \begin{align*}
  \frac{|J^{\pi,m}(s,x)-J^{\pi,m}(t,y)|}
  {(|s-t|+|x-y|^2)^{\alpha/2}}
  &\leq \frac{|V^{\pi,m}(s,s,x)-V^{\pi,m}(t,s,x)|}{(|s-t|+|x-y|^2)^{\alpha/2}}
  +\frac{|V^{\pi,m}(t,s,x)-V^{\pi,m}(t,t,y)|}{(|s-t|+|x-y|^2)^{\alpha/2}}\\
  &\leq \sup_{u\in[0,T]}\|\partial_tV^{\pi,m}(u,\cdot,\cdot)\|_{\Cc^0([u,T]\times\R^d)}+ \sup_{u\in[0,T]}[V^{\pi,m}(u,\cdot,\cdot)]_{\Cc_\alpha([u,T]\times\R^d)}.
 \end{align*}
Notice that $D_x\partial_tV^{\pi,m}=\partial_tD_xV^{\pi,m}$, the same decomposition
 applied to $D_xJ^{\pi,m}$ yields
 $$
 \|J^{\pi,m}\|_{\Cc^{0,1}_\alpha([0,T]\times\R^d)}
 \leq
 \sup_{t\in[0,T]} \|\partial_tV^{\pi,m}(t,\cdot,\cdot)\|_{\Cc^{0,1}([t,T]\times\R^d)}
 + \sup_{t\in[0,T]} \|V^{\pi,m}(t,\cdot,\cdot)\|_{\Cc^{0,1}_\alpha([t,T]\times\R^d)}
 \leq A_0.
 $$
 Increasing $A^*$ if necessary proves the first assertion in \eqref{eq:uniform.Vm}. Moreover, by \eqref{eq:refined.West},
 \begin{align*}
  |D_xV^{\pi,m}(t_1,s,x)-D_xV^{\pi,m}(t_2,s,x)|
  &=\left|\int_{t_2}^{t_1} D_x\partial_tV^{\pi,m}(u,s,x)du\right|\leq A_0|t_1-t_2|,
 \end{align*}
and  proof is completed.
\end{proof}

Fix the constants $A^*,C^*$ in Proposition \ref{prop:uniform.est} and define
$$
\Mc_0:=\left\{w\in\Cc^{0,1}_{\alpha}([0,T]\times\R^d):\|w\|_{\Cc^{0,1}_{\alpha}([0,T]\times\R^d)}\leq A^*\right\},
\qquad
\Ec_0:=\Ac\times\mathscr{M}_\nu^{\kappa,C^*}.
$$
By Proposition \ref{prop:uniform.est}, $J^{\pi,m}\in\Mc_0$ for every $(\pi,m)\in\Ec_0$. Thus $\Mc_0$ records the uniform compactness class of the diagonal value functions, while the fixed-point map itself acts on $\Ec_0$. 

\subsection{Continuity of the value function}

We first record the coefficient convergence implied by convergence in $\Ec_0$.

\begin{Lemma}\label{lm:weakstar.coeff}
Suppose Assumptions \ref{assume.r} and \ref{assume.lipsa.U} hold. Let $(\pi^n,m^n)_{n\in\N\cup\{\infty\}}\subset\Ec_0$ satisfy $\pi^n\to\pi^\infty$ in \eqref{eq:policy.topology} and $d(m^n,m^\infty)\to0$. Then
 \be\label{eq:weakstar}
\tilde b(\cdot,\cdot,m^n_\cdot,\pi^n(\cdot,\cdot))
\to^*
\tilde b(\cdot,\cdot,m^\infty_\cdot,\pi^\infty(\cdot,\cdot))
\quad\text{in }L^\infty([0,T]\times\R^d),
\ee
and, for every fixed $t\in[0,T]$,
$$
\tilde r(\cdot-t,\cdot,m^n_\cdot,\pi^n(\cdot,\cdot))
\to^*\tilde r(\cdot-t,\cdot,m^\infty_\cdot,\pi^\infty(\cdot,\cdot)) \quad\text{in }L^\infty([t,T]\times\R^d).
$$
Moreover,
$$
\|(\sigma\sigma^T)(\cdot,\cdot,m^n_\cdot)-(\sigma\sigma^T)(\cdot,\cdot,m^\infty_\cdot)\|_{L^\infty([0,T]\times\R^d)}\to0.
$$
\end{Lemma}

\begin{proof}
Fix $\phi\in L^1([0,T]\times\R^d)$. Write
$$
\begin{aligned}
&\int_{[0,T]\times\R^d}
\left[\tilde b(t,x,m_t^n,\pi^n(t,x))-\tilde b(t,x,m_t^\infty,\pi^\infty(t,x))\right]\phi(t,x)dtdx\\
=&\int_{[0,T]\times\R^d}\int_U[b(t,x,m_t^n,a)-b(t,x,m_t^\infty,a)]\pi^n(t,x,da)\phi(t,x)dtdx\\
&+\int_{[0,T]\times\R^d}\int_U b(t,x,m_t^\infty,a)[\pi^n(t,x,da)-\pi^\infty(t,x,da)]\phi(t,x)dtdx.
\end{aligned}
$$
The absolute value of the first term is bounded by
$
K_2d(m^n,m^\infty)\|\phi\|_{L^1([0,T]\times\R^d)},
$
which tends to zero as $n\to\infty$. For the second term, use \eqref{eq:policy.topology} with
$$
(t,x,a)\mapsto \frac{\phi(t,x)}{h(x)}b(t,x,m_t^\infty,a).
$$
This is an admissible test function because $h>0$, $a\mapsto b(t,x,m_t^\infty,a)$ is continuous, and
$$
\int_0^T\int_{\R^d}h(x)\max_{a\in U}\left|\frac{\phi(t,x)}{h(x)}b(t,x,m_t^\infty,a)\right|dxdt
\leq K_1\|\phi\|_{L^1}.
$$
This proves \eqref{eq:weakstar}. The proof for $r$ is identical. The convergence of the diffusion coefficient follows directly from Assumption \ref{assume.r}(ii).
\end{proof}

\begin{Lemma}[Stability of the value function]\label{lm:V.stability}
Under the assumptions of Lemma \ref{lm:weakstar.coeff}, for every $N\in\N$ and $0\leq\beta<\alpha$,
 \be\label{eq:V.stability}
\|V^{\pi^n,m^n}-V^{\pi^\infty,m^\infty}\|_{\Cc^{0,1}_{\beta,[0,T]\times B_N(0)}}\to0.
\ee
Furthermore, for every $t\in[0,T]$, every $\phi\in L^q(D_N(t,0))$, and $0\leq2l+|a|_{l^1}\leq2$,
$$
\int_{D_N(t,0)}\phi(s,x)\left(\partial_s^lD_x^aV^{\pi^n,m^n}(t,s,x)-\partial_s^lD_x^aV^{\pi^\infty,m^\infty}(t,s,x)\right)dsdx\to0.
$$
In particular,
$
D_xV^{\pi^n,m^n}(t,t,x)\to D_xV^{\pi^\infty,m^\infty}(t,t,x)
$
locally uniformly in $(t,x)$.
\end{Lemma}

\begin{proof}
 Write
 $$
 V^n:=V^{\pi^n,m^n},\qquad n\in\N\cup\{\infty\}.
 $$
 We prove the result by showing that every subsequence of $(V^n)_{n\in\N}$ contains a further subsequence converging to $V^\infty$ in all the senses stated in the lemma.
 
 \textbf{Step 1.} We first study the convergence of the derivatives..
 By Proposition \ref{prop:uniform.est}, for every $N\in\N$,
 \begin{align}
  &\sup_{n\in\N}
  \left(\sup_{t\in[0,T]}
  \|V^n(t,\cdot,\cdot)\|_{\Cc^{0,1}_\alpha(D_N(t,0))}
  + \sup_{t\in[0,T]}  \|\partial_tV^n(t,\cdot,\cdot)\|_{\Cc^{0,1}_\alpha(D_N(t,0))}
  \right) \leq A^*, \label{eq:refined.stab.1}\\
  &\sup_{n\in\N}  \left( \sup_{t\in[0,T]} \|V^n(t,\cdot,\cdot)\|_{W^{1,2}_p(D_N(t,0))}
  + \sup_{t\in[0,T]} \|\partial_tV^n(t,\cdot,\cdot)\|_{W^{1,2}_p(D_N(t,0))}  \right)
  \leq A_NA^*, \label{eq:refined.stab.2}
 \end{align}
 where $A_N<\infty$ depends only on $d,p,N$.
 
 Fix $0\leq\beta<\alpha$ and start with an arbitrary subsequence of $(V^n)_{n\in\N}$, which we still denote by $(V^n)_{n\in\N}$. We repeat the following construction for $N=T,T+1,\ldots$. Since $\Cc^{0,1}_\alpha(D_N(t,0))$ is compactly embedded in $\Cc^{0,1}_\beta(D_N(t,0))$, and \eqref{eq:refined.stab.1} implies
 $$
 \|V^n(t_1,\cdot,\cdot)-V^n(t_2,\cdot,\cdot)\|_{\Cc^{0,1}(D_N)} \leq A^*|t_1-t_2|,
 $$
 the mappings $t\mapsto V^n(t,\cdot,\cdot)$ are uniformly equicontinuous with values in $\Cc^{0,1}_\beta$ on every bounded cylinder. Hence, by the Arzel\`a--Ascoli theorem, there exist a subsequence, denoted at this stage by $(V^{N,n})_{n\in\N}$, and a function $u^N$ on $\Delta_{[0,T]}\times B_N(0)$ such that
  \be\label{eq:refined.stab.3}
  \lim_{n\to\infty}  \sup_{t\in[0,T]}  \|V^{N,n}(t,\cdot,\cdot)-u^N(t,\cdot,\cdot)\|_{\Cc^{0,1}_\beta(D_N(t,0))}=0.
 \ee
 By \eqref{eq:refined.stab.2} and Banach--Alaoglu, after taking a further subsequence, for every fixed $t\in[0,T]$ there exist $w(t,\cdot,\cdot),w^{ij}(t,\cdot,\cdot)\in L^p(D_N(t,0))$ such that, for every $\phi\in L^q(D_N(t,0))$,
 \begin{align}
  \lim_{n\to\infty}
  \int_{D_N(t,0)}\phi(s,x)\partial_sV^{N,n}(t,s,x)dsdx
  &=  \int_{D_N(t,0)}\phi(s,x)w(t,s,x)dsdx, \label{eq:refined.stab.4}\\
  \lim_{n\to\infty}  \int_{D_N(t,0)}\phi(s,x)\partial_{x_ix_j}^2V^{N,n}(t,s,x)dsdx
  &=  \int_{D_N(t,0)}\phi(s,x)w^{ij}(t,s,x)dsdx. \label{eq:refined.stab.5}
 \end{align}
 We now identify these weak limits. Take $\phi\in C_c^\infty(D_N(t,0))$. By \eqref{eq:refined.stab.3}, $D_xV^{N,n}\to D_xu^N$ uniformly on $D_N(t,0)$, and therefore
 \begin{align*}
  -\int_{D_N(t,0)}\phi\,w^{ij}\,dsdx
  &=-\lim_{n\to\infty}  \int_{D_N(t,0)}\phi\,\partial_{x_ix_j}^2V^{N,n}\,dsdx\\
  &=\lim_{n\to\infty}  \int_{D_N(t,0)}\partial_{x_i}\phi\,  \partial_{x_j}V^{N,n}\,dsdx=\int_{D_N(t,0)}  \partial_{x_i}\phi\,\partial_{x_j}u^N\,dsdx.
 \end{align*}
 Thus $w^{ij}=\partial_{x_ix_j}^2u^N$ in the weak sense. Similarly, using the uniform convergence of $V^{N,n}$ to $u^N$,
 $$
 -\int_{D_N(t,0)}\phi\,w\,dsdx
 =-\lim_{n\to\infty}\int_{D_N(t,0)}\phi\,\partial_sV^{N,n}\,dsdx
 = \int_{D_N(t,0)}\partial_s\phi\,u^N\,dsdx,
 $$
 so $w=\partial_su^N$. Together with the strong convergence of $V^{N,n}$ and $D_xV^{N,n}$, this proves that, for every $0\leq2l+|a|_{l^1}\leq2$ and every $\phi\in L^q(D_N(t,0))$,
  \be\label{eq:refined.stab.6}
  \lim_{n\to\infty}  \int_{D_N(t,0)} \phi(s,x)\left( \partial_s^lD_x^aV^{N,n}(t,s,x) -\partial_s^lD_x^au^N(t,s,x)  \right)dsdx=0.
 \ee
 Proceeding in $N$ with nested subsequences, the limit functions are consistent on overlapping cylinders: $u^{N+1}=u^N$ on $\Delta_{[0,T]}\times B_N(0)$. Hence they define a global function $v^\infty$ on $\Delta_{[0,T]}\times\R^d$. Taking the diagonal subsequence, still denoted by $(V^n)$, we obtain, for every $K\in\N$,
 \begin{align}
  &\lim_{n\to\infty}
  \|V^n-v^\infty\|_{\Cc^{0,1}_{\beta,[0,T]\times B_K(0)}}=0, \label{eq:refined.stab.7}\\
  &\lim_{n\to\infty}  \int_{D_K(t,0)}  \phi(s,x)\left(  \partial_s^lD_x^aV^n(t,s,x)  -\partial_s^lD_x^av^\infty(t,s,x) \right)dsdx=0 \label{eq:refined.stab.8}
 \end{align}
 for every fixed $t\in[0,T]$, $\phi\in L^q(D_K(t,0))$, and $0\leq2l+|a|_{l^1}\leq2$.
 
 We also retain the uniformly local Sobolev estimate in the limit. Fix $(t_0,x_0)\in[t,T]\times\R^d$. By the dual representation of the $L^p$ norm and \eqref{eq:refined.stab.8},
 \begin{align*}
  \|\partial_s^lD_x^av^\infty(t,\cdot,\cdot)\|_{L^p(D_1(t_0,x_0))}
  &=  \sup_{\|\phi\|_{L^q(D_1(t_0,x_0))}\leq1}  \int_{D_1(t_0,x_0)}  \phi\,\partial_s^lD_x^av^\infty\,dsdx\\
  &\leq  \liminf_{n\to\infty}  \|\partial_s^lD_x^aV^n(t,\cdot,\cdot)\|_{L^p(D_1(t_0,x_0))}
  \leq A^*.
 \end{align*}
 Taking the supremum over $(t_0,x_0)$ and then over $t$ gives $\|v^\infty\|_{\widetilde W^{1,2}_{p,[0,T]}}\leq A^*$.
 
 \textbf{Step 2.} We now show that $v^\infty$ solves the equation associated with $(\pi^\infty,m^\infty)$. Fix $t\in[0,T)$, $(t_0,x_0)\in[t,T)\times\R^d$, and $\phi\in L^q(D_1(t_0,x_0))$. Set
 $$
 \tilde b^n(s,x):=\tilde b(s,x,m_s^n,\pi^n(s,x)),\qquad
 \tilde r_t^n(s,x):=\tilde r(s-t,x,m_s^n,\pi^n(s,x)),
 $$
 and define $\tilde b^\infty,\tilde r_t^\infty$ analogously. Since $V^n$ satisfies \eqref{eq:linear.V},
 \begin{align}
  0=\lim_{n\to\infty} \int_{D_1(t_0,x_0)} \bigg[
  &\partial_sV^n +\frac12\tr\left((\sigma\sigma^T)(s,x,m_s^n)D_x^2V^n\right)  +\tilde b^n\cdot D_xV^n+\tilde r_t^n  \bigg]\phi\,dsdx. \label{eq:refined.stab.9}
 \end{align}
 The time derivative converges by \eqref{eq:refined.stab.8}. For the diffusion term, write
 \begin{align*}
  &\int_{D_1}  \tr\left((\sigma\sigma^T)(s,x,m_s^n)D_x^2V^n  -(\sigma\sigma^T)(s,x,m_s^\infty)D_x^2v^\infty\right)\phi\,dsdx\\
  &=  \int_{D_1}  \tr\left((\sigma\sigma^T)(s,x,m_s^\infty)  (D_x^2V^n-D_x^2v^\infty)\right)\phi\,dsdx\\
  &\quad+ \int_{D_1}  \tr\left(  [(\sigma\sigma^T)(s,x,m_s^n)-(\sigma\sigma^T)(s,x,m_s^\infty)]  D_x^2V^n  \right)\phi\,dsdx.
 \end{align*}
 The first term tends to zero by \eqref{eq:refined.stab.8}. For the second term, Assumption \ref{assume.r}(ii) and
 $d(m^n,m^\infty)\to0$ yield
 $$
 \|(\sigma\sigma^T)(\cdot,\cdot,m_\cdot^n) -(\sigma\sigma^T)(\cdot,\cdot,m_\cdot^\infty)\|_{\Cc^0} \leq A_0d(m^n,m^\infty)\to0,
 $$
 while $(D_x^2V^n)$ is uniformly bounded in $L^p(D_1)$; hence H\"older's inequality gives convergence of the second term to zero. For the drift term, decompose
 \begin{align}
  &\int_{D_1}  (\tilde b^n\cdot D_xV^n-\tilde b^\infty\cdot D_xv^\infty)\phi\,dsdx\notag\\
  &=  \int_{D_1}
  \tilde b^n\cdot(D_xV^n-D_xv^\infty)\phi\,dsdx
  +  \int_{D_1}  (\tilde b^n-\tilde b^\infty)\cdot D_xv^\infty\,\phi\,dsdx. \label{eq:refined.stab.10}
 \end{align}
 The first term on the right-hand side tends to zero because $\tilde b^n$ is uniformly bounded and $D_xV^n\to D_xv^\infty$ uniformly on $D_1$. Since $D_xv^\infty\,\phi\in L^1(D_1)$, the second term tends to zero by Lemma \ref{lm:weakstar.coeff}. The reward term satisfies
 $$
 \int_{D_1}(\tilde r_t^n-\tilde r_t^\infty)\phi\,dsdx\to0
 $$
 by the same weak-$*$ convergence. Passing to the limit in
 \eqref{eq:refined.stab.9}, we obtain
  \be\label{eq:refined.stab.limitPDE}
  \partial_sv^\infty(t,s,x)  +\frac12\tr\left((\sigma\sigma^T)(s,x,m_s^\infty)D_x^2v^\infty(t,s,x)\right)
  +\tilde b^\infty(s,x)\cdot D_xv^\infty(t,s,x)  +\tilde r_t^\infty(s,x)=0
 \ee
 a.e. on $[t,T)\times\R^d$.
 
 The terminal condition also passes to the limit. Indeed, \eqref{eq:refined.stab.7} gives $V^n(t,T,x)\to v^\infty(t,T,x)$ locally uniformly in $x$, whereas Assumption \ref{assume.r}(iv) and $d(m^n,m^\infty)\to0$ give
 $$
 F(t,x,m_T^n)\to F(t,x,m_T^\infty)
 $$
 locally uniformly in $(t,x)$. Since $V^n(t,T,x)=F(t,x,m_T^n)$, $ v^\infty(t,T,x)=F(t,x,m_T^\infty).$
 
 \textbf{Step 3.} Now we show $v^\infty= V^{\pi^\infty,m^\infty}$.
 Fix $(t,s,x)\in\Delta_{[0,T]}\times\R^d$ and let $X^\infty$ solve
 $$
 dX_l^\infty = \tilde b(l,X_l^\infty,m_l^\infty,\pi^\infty(l,X_l^\infty))dl +\sigma(l,X_l^\infty,m_l^\infty)dW_l,\qquad X_s^\infty=x.
 $$
 By boundedness of the drift, uniform ellipticity, and the regularity of $\sigma$, this SDE admits a unique strong solution. By Krylov's estimate, the law of $X_l^\infty$ admits a density $p_{s,x}(l,y)\in L^q([s,T]\times\R^d)$. 
 Let $\rho_{s,N}:=\inf\{l\geq s:X_l^\infty\notin B_N(0)\}$. Since the coefficients are bounded,
 $$
 \sup_{l\in[s,T]}\E_{s,x}|X_l^\infty|^2<\infty,\quad
 \P_{s,x}(\rho_{s,N}\leq T)
 \leq\frac{1}{N^2} \E_{s,x}\left[\sup_{l\in[s,T]}|X_l^\infty|^2\right]\to0.
 $$
 Thus, for every $\varepsilon>0$, there exists $N_0$ such that, for
 $N\geq N_0$,
 \begin{align}
  &\E_{s,x}\left[  \left|v^\infty(t,T\wedge\rho_{s,N},X^\infty_{T\wedge\rho_{s,N}}) -F(t,X_T^\infty,m_T^\infty)\right| \mathbf{1}_{\{\rho_{s,N}<T\}}  \right]\notag\\
  &\quad+  \E_{s,x}\left[  \int_{\rho_{s,N}}^T  |\tilde r(l-t,X_l^\infty,m_l^\infty,\pi^\infty(l,X_l^\infty))|dl  \mathbf{1}_{\{\rho_{s,N}<T\}}  \right]\leq\varepsilon. \label{eq:refined.stab.tail}
 \end{align} 
 For $N\geq N_0$, apply the It\^o--Krylov formula to $v^\infty(t,l,X_l^\infty)$ on $[s,T\wedge\rho_{s,N}]$. Taking expectations, and using that the stopped transition density belongs to $L^q(D_N(s,0))$, we obtain
 \begin{align*}
  &\E_{s,x}\left[  v^\infty(t,T\wedge\rho_{s,N},X^\infty_{T\wedge\rho_{s,N}})  \right]-v^\infty(t,s,x)\\
  &=  \int_{D_N(s,0)}  \bigg[  \partial_sv^\infty(t,l,y)
  +\tilde b^\infty(l,y)\cdot D_xv^\infty(t,l,y)
  +\frac12\tr\left((\sigma\sigma^T)(l,y,m_l^\infty)D_x^2v^\infty(t,l,y)\right)  \bigg]p_{s,x}(l,y)dydl\\
  &=  -\int_{D_N(s,0)}  \tilde r(l-t,y,m_l^\infty,\pi^\infty(l,y))  p_{s,x}(l,y)dydl\\
  &=  -\E_{s,x}\left[  \int_s^{T\wedge\rho_{s,N}}  \tilde r(l-t,X_l^\infty,m_l^\infty,\pi^\infty(l,X_l^\infty))dl  \right],
 \end{align*}
 where the second equality follows from \eqref{eq:refined.stab.limitPDE}. Therefore,
 \begin{align*}
  v^\infty(t,s,x)
  &=  \E_{s,x}\left[  v^\infty(t,T\wedge\rho_{s,N},X^\infty_{T\wedge\rho_{s,N}})
  +\int_s^{T\wedge\rho_{s,N}}  \tilde r(l-t,X_l^\infty,m_l^\infty,\pi^\infty(l,X_l^\infty))dl  \right].
 \end{align*}
 Combining this identity with \eqref{eq:refined.stab.tail} and then letting $N\to\infty$ gives
 $$
 v^\infty(t,s,x)
 = \E_{s,x}\left[ \int_s^T \tilde r(l-t,X_l^\infty,m_l^\infty,\pi^\infty(l,X_l^\infty))dl
 +F(t,X_T^\infty,m_T^\infty) \right]
 = V^{\pi^\infty,m^\infty}(t,s,x).
 $$ 
 Thus every subsequence of $(V^n)$ has a further subsequence converging to the same limit $V^{\pi^\infty,m^\infty}$ in \eqref{eq:refined.stab.7}--\eqref{eq:refined.stab.8}. The standard subsequence principle gives convergence of the whole sequence, proving \eqref{eq:V.stability} and the weak convergence of all derivatives stated in the lemma. In particular,
 $$
 \sup_{\substack{t\in[0,T],\,x\in B_N(0)}} |D_xV^{\pi^n,m^n}(t,t,x)-D_xV^{\pi^\infty,m^\infty}(t,t,x)| \to0,\quad N\in\N.
 $$
\end{proof}

\subsection{Continuity of the population-flow update}

The weak-$*$ convergence \eqref{eq:weakstar} is not sufficient by itself to apply standard strong SDE stability. We therefore use a duality argument based on the Fokker--Planck equations.

\begin{Proposition}\label{prop:flow.continuity}
Suppose Assumptions \ref{assume.r} and \ref{assume.lipsa.U} hold. Let $(\pi^n,m^n)_{n\in\N\cup\{\infty\}}\subset\Ec_0$ satisfy $\pi^n\to\pi^\infty$ in \eqref{eq:policy.topology} and $d(m^n,m^\infty)\to0$. Then
 \be\label{eq:flow.continuity}
d(\mu^{\pi^n,m^n},\mu^{\pi^\infty,m^\infty})\to0.
\ee
\end{Proposition}

\begin{proof}
 Write
 $$
 \mu^n:=\mu^{\pi^n,m^n},\quad 
 \tilde b^n(t,x):=\tilde b(t,x,m_t^n,\pi^n(t,x)),\quad
 \sigma^n(t,x):=\sigma(t,x,m_t^n),
 $$
 for $n\in\N\cup\{\infty\}$. By Proposition \ref{prop:uniform.est}, $(\mu^n)_{n\in\N}\subset\mathscr{M}_\nu^{\kappa,C^*}$. Since $\mathscr{M}_\nu^{\kappa,C^*}$ is compact in $(\mathscr{M},d)$, every subsequence of $(\mu^n)$ admits a further subsequence, still denoted by $(\mu^n)$, and a flow $\bar\mu\in\mathscr{M}_\nu^{\kappa,C^*}$ such that $d(\mu^n,\bar\mu)\to0$. We prove that $\bar\mu=\mu^\infty:=\mu^{\pi^\infty,m^\infty}$.
 
 Fix $s\in(0,T]$ and $\phi\in C_c^2(\R^d)$. Consider the backward PDE
  \be\label{eq:refined.flow.backward}
  \begin{cases}
   \displaystyle
   \partial_tu(t,x)   +\frac12\tr\left(   (\sigma^\infty\sigma^{\infty,T})(t,x)D_x^2u(t,x)   \right)  
   +\tilde b^\infty(t,x)\cdot D_xu(t,x)=0,
   &(t,x)\in[0,s)\times\R^d,\\
     u(s,x)=\phi(x),&x\in\R^d.
  \end{cases}
 \ee
 By Assumption \ref{assume.r}, the uniform ellipticity and the $\Cc_\alpha$ regularity of $\sigma^\infty$ established in \eqref{eq:refined.sigma.holder}, equation \eqref{eq:refined.flow.backward} admits a unique strong solution $u\in W^{1,2}_p([0,s]\times\R^d)$. The parabolic Sobolev embedding gives $D_xu\in C([0,s]\times\R^d)$ and $\|D_xu\|_{L^\infty([0,s]\times\R^d)}<\infty$. Let $X^n$ solve
 $$
 dX_t^n=\tilde b^n(t,X_t^n)dt+\sigma^n(t,X_t^n)dW_t,\quad X_0^n=\xi\sim\nu.
 $$
 Applying the It\^o--Krylov formula to $u(t,X_t^n)$ and using
 \eqref{eq:refined.flow.backward}, we obtain
 \begin{align}
  \E[\phi(X_s^n)]-\E[u(0,\xi)]
  &=  \E\int_0^s  (\tilde b^n-\tilde b^\infty)(t,X_t^n)\cdot D_xu(t,X_t^n)dt \notag\\
  &\quad+  \frac12\E\int_0^s  \tr\left(  [(\sigma^n\sigma^{n,T})-(\sigma^\infty\sigma^{\infty,T})]  (t,X_t^n)D_x^2u(t,X_t^n)  \right)dt. \label{eq:refined.flow.Ito}
 \end{align}
 
 We next analyze the first term in the right-hand side of \eqref{eq:refined.flow.Ito}. For every $t>0$, uniform ellipticity implies that $X_t^n$ admits a density
 $p^n(t,\cdot)$. Fix $\tau\in(0,s)$. On
 $[\tau,s]\times\R^d$, $p^n$ is a generalized solution of the
 divergence-form equation
 \begin{align}
  \partial_tp^n  -\sum_{i=1}^d\partial_{x_i}\bigg(
  \frac12\sum_{j=1}^d  (\sigma^n\sigma^{n,T})_{ij}\partial_{x_j}p^n
  +  \Big[  \frac12\sum_{j=1}^d  \partial_{x_j}(\sigma^n\sigma^{n,T})_{ij}  -\tilde b_i^n \Big]p^n  \bigg)=0. \label{eq:refined.flow.FP}
 \end{align}
 The coefficients in \eqref{eq:refined.flow.FP} satisfy bounds which are uniform in $n$: $\sigma^n\sigma^{n,T}$ is uniformly elliptic, $\sigma^n$ is uniformly Lipschitz in $x$ by Assumption \ref{assume.r}(ii), and $\tilde b^n$ is uniformly bounded. Consequently, the local boundedness and interior H\"older estimates for generalized solutions of divergence-form parabolic equations imply that for every $\tau\in(0,s)$ and $R>0$ there exist $\gamma\in(0,1)$ and $C_{\tau,R}<\infty$, independent of $n$, such that
  \be\label{eq:refined.flow.localholder}
  \sup_{n\in\N} \|p^n\|_{\Cc_\gamma([\tau,s]\times B_R(0))}  \leq C_{\tau,R}.
 \ee
 By Arzel\`a--Ascoli and a diagonal argument, after passing to a further subsequence if necessary, there exists a nonnegative continuous function $p^*$ on $(0,s]\times\R^d$ such that
  \be\label{eq:refined.flow.localunif}
  p^n\to p^*  \quad\text{locally uniformly on }(0,s]\times\R^d.
 \ee
 Since the uniform moment estimate implies that $\{p^n(t,x)\,dt\,dx\}_{n\in\N}$ is tight on $[\tau,s]\times\R^d$, the locally uniform convergence in
 \eqref{eq:refined.flow.localunif} yields
$$
 \int_\tau^s\int_{\R^d}p^*(t,x) dxdt =
 \lim_{n\to\infty} \int_\tau^s\int_{\R^d}p^n(t,x) dxdt
 =s-\tau.
$$
Therefore, by Scheff\'e's lemma,
  \be\label{eq:refined.flow.L1}
 \|p^n-p^*\|_{L^1([\tau,s]\times\R^d)}\to0.
\ee
 
 Now decompose the first term in the right-hand side of \eqref{eq:refined.flow.Ito} as
 \begin{align}
  &\E\int_0^s  (\tilde b^n-\tilde b^\infty)(t,X_t^n)\cdot D_xu(t,X_t^n)dt \notag\\
  &=  \E\int_0^\tau  (\tilde b^n-\tilde b^\infty)(t,X_t^n)\cdot D_xu(t,X_t^n)dt 
 +  \int_\tau^s\int_{\R^d}  (\tilde b^n-\tilde b^\infty)(t,x)\cdot D_xu(t,x) p^*(t,x)dxdt \notag\\
  &\quad+  \int_\tau^s\int_{\R^d}  (\tilde b^n-\tilde b^\infty)(t,x)\cdot D_xu(t,x)  [p^n(t,x)-p^*(t,x)]dxdt. \label{eq:refined.flow.driftdec}
 \end{align}
 Because $\tilde b^n$ and $D_xu$ are uniformly bounded, the absolute value of the first term on the right-hand side is at most $A_0\tau$. For the second term, $D_xu\,p^*\mathbf{1}_{[\tau,s]}\in L^1([0,T]\times\R^d)$, and Lemma \ref{lm:weakstar.coeff} gives
 $$
 \int_\tau^s\int_{\R^d} (\tilde b^n-\tilde b^\infty)\cdot D_xu\,p^*\,dxdt\to0.
 $$
 By \eqref{eq:refined.flow.L1}, the third term is estimated by
$
 A_0\|D_xu\|_{L^\infty([0,T]\times \R^d)} \|p^n-p^*\|_{L^1([\tau,s]\times\R^d)}\to0.
$ Therefore,
  \be\label{eq:refined.flow.driftzero}
  \limsup_{n\to\infty}  \left|  \E\int_0^s  (\tilde b^n-\tilde b^\infty)(t,X_t^n)\cdot D_xu(t,X_t^n)dt  \right|  \leq A_0\tau.
 \ee
 Letting $\tau\downarrow0$ proves that the drift contribution tends to zero.
 
 For the diffusion contribution, Assumption \ref{assume.r}(ii) and $d(m^n,m^\infty)\to0$ imply
  \be\label{eq:refined.flow.sigmaconv}
  \|\sigma^n\sigma^{n,T}-\sigma^\infty\sigma^{\infty,T}\|_{L^\infty([0,s]\times\R^d)} \leq A_0d(m^n,m^\infty)\to0.
 \ee
 Krylov's estimate gives $\sup_n\|p^n\|_{L^q([0,s]\times\R^d)}<\infty$. Hence H\"older's inequality yields
 \begin{align*}
  &\left|  \E\int_0^s\tr\left(  [(\sigma^n\sigma^{n,T})-(\sigma^\infty\sigma^{\infty,T})]  (t,X_t^n)D_x^2u(t,X_t^n)  \right)dt  \right|\\
  &\quad\leq  \|\sigma^n\sigma^{n,T}-\sigma^\infty\sigma^{\infty,T}\|_\infty  \|D_x^2u\|_{L^p([0,s]\times\R^d)}  \|p^n\|_{L^q([0,s]\times\R^d)}  \to0. 
 \end{align*}

In sum, by passing to the limit in \eqref{eq:refined.flow.Ito}, we have
 $$
 \lim_{n\to\infty}\E[\phi(X_s^n)]=\E[u(0,\xi)].
 $$
 Applying the It\^o--Krylov formula to $u(t,X_t^\infty)$ and using \eqref{eq:refined.flow.backward} gives $\E[\phi(X_s^\infty)]=\E[u(0,\xi)]$. Therefore,
  \be\label{eq:refined.flow.weaklaw}
  \lim_{n\to\infty}\E[\phi(X_s^n)]  =  \E[\phi(X_s^\infty)],  \quad \phi\in C_c^2(\R^d).
 \ee
 On the other hand, $d(\mu^n,\bar\mu)\to0$ implies
 $$
 \lim_{n\to\infty} \E[\phi(X_s^n)] = \lim_{n\to\infty}\int_{\R^d}\phi(x)\mu_s^n(dx) = \int_{\R^d}\phi(x)\bar\mu_s(dx).
 $$
 Combining this with \eqref{eq:refined.flow.weaklaw} yields
 $$
 \int_{\R^d}\phi(x)\bar\mu_s(dx)
 = \int_{\R^d}\phi(x)\mu_s^\infty(dx), \quad \phi\in C_c^2(\R^d).
 $$
 Thus $\bar\mu_s=\mu_s^\infty$ for every $s\in(0,T]$, and the equality at $s=0$ follows from the common initial law $\nu$. Hence $\bar\mu=\mu^\infty$. Since every subsequence has the same unique subsequential limit, the whole sequence satisfies $d(\mu^n,\mu^\infty)\to0$.
\end{proof}

\subsection{The set-valued mapping and the existence of its fixed point}

For $(\pi,m)\in\Ec_0$ and $\varpi\in\Ac$, define
$$
\mathcal{G}(\pi,m;\varpi)
:=\int_0^T\int_{\R^d}h(x)
\int_U\left[b(t,x,m_t,a)\cdot D_xV^{\pi,m}(t,t,x)+r(0,x,m_t,a) \right]\varpi(t,x,da)dxdt.
$$
Define the set-valued mapping $\Phi:\Ec_0\to\Ec_0$ by
 \be\label{eq:Phi.direct}
\Phi(\pi,m)
:=\left\{(\varpi,\mu^{\pi,m}):\varpi\in\argmax_{\tilde\pi\in\Ac}\mathcal{G}(\pi,m;\tilde\pi)
\right\}.
\ee

\begin{Lemma}\label{lm:G.continuous}
Suppose Assumptions \ref{assume.r} and \ref{assume.lipsa.U} hold. The functional $(\pi,m,\varpi)\mapsto\mathcal{G}(\pi,m;\varpi)$ is jointly continuous on $\Ec_0\times\Ac$.
\end{Lemma}

\begin{proof}
 Let
 $$
 (\pi^n,m^n,\varpi^n)\to (\pi^\infty,m^\infty,\varpi^\infty) \quad\text{in }\Ec_0\times\Ac.
 $$
 We prove
 $\mathcal{G}(\pi^n,m^n;\varpi^n)\to
 \mathcal{G}(\pi^\infty,m^\infty;\varpi^\infty)$.
 For notational convenience, set
 $$
 D_xJ^n(t,x):=D_xV^{\pi^n,m^n}(t,t,x),\quad
 D_xJ^\infty(t,x):=D_xV^{\pi^\infty,m^\infty}(t,t,x).
 $$
 By Lemma \ref{lm:V.stability},
 $D_xJ^n(t,x)\to D_xJ^\infty(t,x)$ for every $(t,x)$, locally uniformly,
 and Proposition \ref{prop:uniform.est} gives
  \be\label{eq:refined.G.gradbound}
  \sup_{n\in\N\cup\{\infty\}}  \|D_xJ^n\|_{L^\infty([0,T]\times\R^d)}  \leq A^*.
 \ee
 
 We first consider the drift part. Adding and subtracting the two
 intermediate terms in which, successively, the gradient and the measure
 argument are replaced, we have
 \begin{align}
  &\int_0^T\int_{\R^d}h(x)  \tilde b(t,x,m_t^n,\varpi^n(t,x))\cdot D_xJ^n(t,x)dxdt\notag\\
  &\quad- \int_0^T\int_{\R^d}h(x)  \tilde b(t,x,m_t^\infty,\varpi^\infty(t,x)) \cdot D_xJ^\infty(t,x)dxdt\notag\\
  &=:\text{I}_n+\text{II}_n+\text{III}_n, \label{eq:refined.G.decomp}
 \end{align}
 where
 \begin{align*}
  \text{I}_n
  &:=  \int_0^T\int_{\R^d}h(x)  \tilde b(t,x,m_t^n,\varpi^n(t,x))
  \cdot[D_xJ^n(t,x)-D_xJ^\infty(t,x)]dxdt,\\
  \text{II}_n
  &:=  \int_0^T\int_{\R^d}h(x)  \left[  \tilde b(t,x,m_t^n,\varpi^n(t,x))  -\tilde b(t,x,m_t^\infty,\varpi^n(t,x))  \right]\cdot D_xJ^\infty(t,x)dxdt,\\
  \text{III}_n
  &:= \int_0^T\int_{\R^d}h(x) \left[  \tilde b(t,x,m_t^\infty,\varpi^n(t,x)) -\tilde b(t,x,m_t^\infty,\varpi^\infty(t,x)) \right]\cdot D_xJ^\infty(t,x)dxdt.
 \end{align*}
 For $\text{I}_n$, pointwise convergence of the gradients, the uniform bounds
 on $b$ and \eqref{eq:refined.G.gradbound}, and
 $\int_{\R^d}h(x)dx=1$ allow us to apply dominated convergence:
 $$
 |\text{I}_n| \leq \|b\|_\infty \int_0^T\int_{\R^d} h(x)|D_xJ^n(t,x)-D_xJ^\infty(t,x)|dxdt \to0.
 $$
 For $\text{II}_n$, Assumption \ref{assume.r}(i) gives
 \begin{align*}
  |\text{II}_n|  &\leq K_2\int_0^T\int_{\R^d}  h(x)\Wc_2(m_t^n,m_t^\infty)  |D_xJ^\infty(t,x)|dxdt\leq K_2TA^*d(m^n,m^\infty)\to0.
 \end{align*}
 For $\text{III}_n$, define the Carath\'eodory test function
 $$
 \phi_b(t,x,a):=b(t,x,m_t^\infty,a)\cdot D_xJ^\infty(t,x).
 $$
 It is continuous in $a$, and by \eqref{eq:refined.G.gradbound},
 $$
 \int_0^T\int_{\R^d} h(x)\max_{a\in U}|\phi_b(t,x,a)|dxdt \leq T\|b\|_\infty A^*<\infty.
 $$
 Therefore, the convergence $\varpi^n\to\varpi^\infty$ in
 \eqref{eq:policy.topology} yields
 \begin{align*}
  \text{III}_n &=  \int_0^T\int_{\R^d}h(x) \left( \int_U\phi_b(t,x,a)\varpi^n(t,x,da) -\int_U\phi_b(t,x,a)\varpi^\infty(t,x,da)  \right)dxdt \to0.
 \end{align*}
 Thus the drift part of $\mathcal{G}$ converges. 
  The running reward part can be treated in a similar way. Combining the drift and reward estimates gives
 $$
 \mathcal{G}(\pi^n,m^n;\varpi^n) \to \mathcal{G}(\pi^\infty,m^\infty;\varpi^\infty),
 $$
 which proves the joint continuity.
\end{proof}

\begin{Theorem}[Existence of a fixed point]\label{thm:fixedpoint}
Let Assumptions \ref{assume.r} and \ref{assume.lipsa.U} hold and suppose $\nu\in\Pc_{2+\kappa}(\R^d)$. There exists $(\pi^*,m^*)\in\Ec_0$ such that
$$
(\pi^*,m^*)\in\Phi(\pi^*,m^*).
$$
In particular,
 \be\label{eq:fixedpoint.conditions}
\pi^*\in\argmax_{\varpi\in\Ac}\mathcal{G}(\pi^*,m^*;\varpi),
\qquad
m^*=\mu^{\pi^*,m^*}.
\ee
\end{Theorem}

\begin{proof}
We first show that The set $\Ec_0$ is compact and convex. The compactness and convexity of $\Ac$ follow from Lemma \ref{lm:policy.compact}. Also, $\mathscr{M}_\nu^{\kappa,C^*}$ is a compact convex subset of a locally convex topological vector space, with the same topology as the metric $d$ on this set. Hence $\Ec_0=\Ac\times\mathscr{M}_\nu^{\kappa,C^*}$ is a compact convex subset of a locally convex topological vector space.

Now we show the properties of mapping $\Phi$. Fix $(\pi,m)\in\Ec_0$. By Lemma \ref{lm:G.continuous}, $\varpi\mapsto\mathcal{G}(\pi,m;\varpi)$ is continuous on the compact set $\Ac$, so the maximum is attained. Hence $\Phi(\pi,m)$ is nonempty. Since $\mathcal{G}(\pi,m;\cdot)$ is affine, its set of maximizers is convex. It is also closed in the compact set $\Ac$, hence compact. The second component of every element of $\Phi(\pi,m)$ is the singleton $\{\mu^{\pi,m}\}$, which is convex and compact. Thus $\Phi(\pi,m)$ is compact and convex. Moreover, $\Phi$ is upper hemicontinuous. Indeed, let $(\pi^n,m^n)\to(\pi,m)$ and $(\varpi^n,\mu^{\pi^n,m^n})\in\Phi(\pi^n,m^n)$ satisfy
$$
(\varpi^n,\mu^{\pi^n,m^n})\to(\varpi,\mu).
$$
Proposition \ref{prop:flow.continuity} gives $\mu=\mu^{\pi,m}$. For any $\widehat\pi\in\Ac$,
$$
\mathcal{G}(\pi^n,m^n;\varpi^n)\geq\mathcal{G}(\pi^n,m^n;\widehat\pi).
$$
Passing to the limit by Lemma \ref{lm:G.continuous} yields
$$
\mathcal{G}(\pi,m;\varpi)\geq\mathcal{G}(\pi,m;\widehat\pi).
$$
Thus $\varpi$ is a maximizer and $(\varpi,\mu^{\pi,m})\in\Phi(\pi,m)$. Hence the graph of $\Phi$ is closed. Since the domain and range are compact, $\Phi$ is upper hemicontinuous.	
	
Since $\Ec_0$ is a nonempty compact convex subset of a locally convex topological vector space, and $\Phi:\Ec_0\to\Ec_0$ is upper hemicontinuous with nonempty compact convex values. The Kakutani--Fan--Glicksberg fixed-point theorem therefore yields a fixed point $(\pi^*,m^*)\in\Ec_0$. The two identities in \eqref{eq:fixedpoint.conditions} follow immediately from the definition \eqref{eq:Phi.direct}.
\end{proof}

\subsection{Characterization of  an equilibrium as a fixed point}

We first convert the integrated maximizing property in \eqref{eq:fixedpoint.conditions} into the pointwise Hamiltonian condition. Define
\be\label{eq:def.hamiltonian} 
\begin{aligned}
H_0(t,s, x,a)  &:=b(s,x,m_t^*,a)\cdot D_xV^{\pi^*,m^*}(t,s,x)+r(s-t,x,m_t^*,a),\\
  H(t,s,x)  &:=  \sup_{\varpi\in\Pc(U)}  \int_u H_0(t,s, x,a)  \varpi(da). 
\end{aligned}
\ee 
\begin{Proposition}\label{prop:pointwise.max}
Let $(\pi^*,m^*)$ be a fixed point in Theorem \ref{thm:fixedpoint}. Then
 \be\label{eq:pointwise.max}
\begin{aligned}
&\tilde b(t,x,m_t^*,\pi^*(t,x))\cdot D_xV^{\pi^*,m^*}(t,t,x)
+\tilde r(0,x,m_t^*,\pi^*(t,x))\\
&=\sup_{\varpi\in\Pc(U)}\left\{\int_U[b(t,x,m_t^*,a)\cdot D_xV^{\pi^*,m^*}(t,t,x)+r(0,x,m_t^*,a)]\varpi(da)\right\}
\end{aligned}
\ee
for a.e. $(t,x)\in[0,T]\times\R^d$.
\end{Proposition}

\begin{proof}
Recall \eqref{eq:def.hamiltonian}.
The function $(t,x,a)\mapsto H_0(t,t,x,a)$ is Borel measurable in $(t,x)$ and continuous in $a$. Since $U$ is compact, the measurable maximum theorem yields a measurable selector $a^*(t,x)\in U$ such that
$$
H_0(t,t,x,a^*(t,x))=\max_{a\in U}H_0(t,t,x,a).
$$
Set $\widehat\pi(t,x):=\delta_{a^*(t,x)}$. By the fixed-point property,
$$
\mathcal{G}(\pi^*,m^*;\pi^*)\geq\mathcal{G}(\pi^*,m^*;\widehat\pi).
$$
On the other hand,
$
\int_UH_0(t,t,x,a)\pi^*(t,x,da)\leq H_0(t,t,x,a^*(t,x))
$
for every $(t,x)$. Therefore
$$
\int_0^T\int_{\R^d}h(x)\left[\max_{a\in U}H_0(t, t,x,a)-\int_UH_0(t,t,x,a)\pi^*(t,x,da)\right]dxdt=0.
$$
The integrand is nonnegative and $h(x)>0$. It follows that equality holds a.e., which is exactly \eqref{eq:pointwise.max}.
\end{proof}

\begin{Proposition}[Strong EHJB and verification]\label{prop:verification}
Let $(\pi^*,m^*)$ be a fixed point in Theorem \ref{thm:fixedpoint}. Then $V^{\pi^*,m^*}\in\widetilde{\Cc}^{0,1}_{\alpha,[0,T]}\cap\widetilde W^{1,2}_{p,[0,T]}$ is a strong solution of
\be\label{eq:EHJB.linear}
\begin{cases}
\begin{aligned}
&0=\partial_sV^{\pi^*,m^*}(t,s,x)+\frac12\tr((\sigma\sigma^T)(s,x,m_s^*)D_x^2V^{\pi^*,m^*}(t,s,x))\\
&\qquad +\tilde b(s,x,m_s^*,\pi^*(s,x))\cdot D_xV^{\pi^*,m^*}(t,s,x)+\tilde r(s-t,x,m_s^*,\pi^*(s,x)),\\
&V^{\pi^*,m^*}(t,T,x)=F(t,x,m_T^*),
\end{aligned}
\end{cases}
\ee
and $\pi^*$ satisfies the equilibrium response condition \eqref{eq:def.equipi}.
\end{Proposition}

\begin{proof}
 The strong-solution property in \eqref{eq:EHJB.linear} follows from
 Proposition \ref{prop:uniform.est}. We prove the equilibrium response
 condition \eqref{eq:def.equipi} in several steps.
 
 \textbf{Step 1.} 
 For $(t,s,x)\in\Delta_{[0,T]}\times\R^d$, define
 \begin{align*}
  L(t,s,x)  &:=  \tilde b(s,x,m_s^*,\pi^*(s,x))  \cdot D_xV^{\pi^*,m^*}(t,s,x)  +\tilde r(s-t,x,m_s^*,\pi^*(s,x)),
 \end{align*}
and recall \eqref{eq:def.hamiltonian}.
We first verify a lower bound for $L-H$. We observe that there exists $C<\infty$ such that
  \be\label{eq:refined.verify.H}
  \sup_{x\in\R^d}|H(t,s,x)-H(s,s,x)|
  \leq C|t-s|^{\alpha/2},\qquad s\in[t,T].
 \ee
 Indeed, fix $(t,s,x)$, by compactness of $U$ and continuity in the action
 variable, there exists $\hat a\in U$ satisfying
 $$
 H(t,s,x)= b(s,x,m_s^*,\hat a)\cdot D_xV^{\pi^*,m^*}(t,s,x) +r(s-t,x,m_s^*,\hat a).
 $$
 By \eqref{eq:uniform.tlip} and the H\"older continuity of $r$ in its first time argument,
 \begin{align*}
  H(t,s,x)  &\leq  b(s,x,m_s^*,\hat a)\cdot  D_xV^{\pi^*,m^*}(s,s,x)  +r(0,x,m_s^*,\hat a)  +C|t-s|^{\alpha/2}\\
  &\leq H(s,s,x)+C|t-s|^{\alpha/2}.
 \end{align*}
 Interchanging the roles of $H(t,s,x)$ and $H(s,s,x)$ gives the reverse
 inequality and proves \eqref{eq:refined.verify.H}. Similarly,
  \be\label{eq:refined.verify.L}
  \sup_{x\in\R^d}|L(t,s,x)-L(s,s,x)|  \leq C|t-s|^{\alpha/2},\qquad s\in[t,T].
 \ee
 By Proposition \ref{prop:pointwise.max},
 $L(s,s,x)=H(s,s,x)$ for a.e. $(s,x)$. Consequently,
 \begin{align}
  L(t,s,x)-H(t,s,x)
  &=[L(t,s,x)-L(s,s,x)]  +[L(s,s,x)-H(s,s,x)]\notag\\
  &\quad+[H(s,s,x)-H(t,s,x)]  \geq-2C|t-s|^{\alpha/2} \label{eq:refined.verify.HL}
 \end{align}
 for a.e. $(s,x)\in[t,T]\times\R^d$.
 
 \textbf{Step 2.} 
Now we truncate the value functions. Fix an arbitrary $(t,x)\in[0,T)\times\R^d$ and
 $\pi'\in\Ac$. For $\varepsilon_0\in(0,T-t]$, let
 $X^{\pi'}$ solve, on $[t,t+\varepsilon_0]$,
 $$
 dX_l^{\pi'} = \tilde b(l,X_l^{\pi'},m_l^*,\pi'(l,X_l^{\pi'}))dl +\sigma(l,X_l^{\pi'},m_l^*)dW_l,\qquad X_t^{\pi'}=x.
 $$
 Define $\rho_{t,N}:=\inf\{l\geq t:X_l^{\pi'}\notin B_N(0)\}$. By boundedness of the coefficients,
 $$
 \P_{t,x}(\rho_{t,N}\leq t+\varepsilon_0) \leq \frac{1}{N^2} \E_{t,x}\left[ \sup_{l\in[t,t+\varepsilon_0]}|X_l^{\pi'}|^2 \right]\longrightarrow0.
 $$
 Let $A_0$ be a uniform bound for $|V^{\pi^*,m^*}|$, $|r|$, and $|F|$. Given $\varepsilon>0$, choose $N_0$ sufficiently large such that, for every $N\geq N_0$,
 \begin{align}
  &\left|  \E_{t,x}\left[  V^{\pi^*,m^*}(t,t+\varepsilon_0,X_{t+\varepsilon_0}^{\pi'})
  -  V^{\pi^*,m^*}  (t,(t+\varepsilon_0)\wedge\rho_{t,N},  X_{(t+\varepsilon_0)\wedge\rho_{t,N}}^{\pi'})  \right]  \right|\notag\\
  &\quad+  \left|  \E_{t,x}\left[  \int_t^{t+\varepsilon_0}  \tilde r(l-t,X_l^{\pi'},m_l^*,\pi'(l,X_l^{\pi'}))dl
  -  \int_t^{(t+\varepsilon_0)\wedge\rho_{t,N}}  \tilde r(l-t,X_l^{\pi'},m_l^*,\pi'(l,X_l^{\pi'}))dl  \right]  \right|
  \leq\varepsilon. \label{eq:refined.verify.localerr}
 \end{align}
 Indeed, both differences vanish on
 $\{\rho_{t,N}>t+\varepsilon_0\}$ and are bounded by a fixed constant on
 the complementary event. Notice that
 \begin{align}
  J^{\pi'\otimes_{t,\varepsilon_0}\pi^*,m^*}(t,x)
  =  \E_{t,x}\bigg[  \int_t^{t+\varepsilon_0}  \tilde r(l-t,X_l^{\pi'},m_l^*,\pi'(l,X_l^{\pi'}))dl
  +  V^{\pi^*,m^*}(t,t+\varepsilon_0,X_{t+\varepsilon_0}^{\pi'})
  \bigg]. \label{eq:refined.verify.concat}
 \end{align}
 Combining \eqref{eq:refined.verify.localerr} and
 \eqref{eq:refined.verify.concat}, for every $N\geq N_0$,
 \begin{align}
 J^{\pi'\otimes_{t,\varepsilon_0}\pi^*,m^*}(t,x)  -J^{\pi^*,m^*}(t,x)
  &\leq  \E_{t,x}\bigg[  V^{\pi^*,m^*}  (t,(t+\varepsilon_0)\wedge\rho_{t,N},  X_{(t+\varepsilon_0)\wedge\rho_{t,N}}^{\pi'})  -  V^{\pi^*,m^*}(t,t,x)\notag\\
  &\qquad\qquad   +\int_t^{(t+\varepsilon_0)\wedge\rho_{t,N}}  \tilde r(l-t,X_l^{\pi'},m_l^*,\pi'(l,X_l^{\pi'}))dl  \bigg]+\varepsilon. \label{eq:refined.verify.beforeIto}
 \end{align}
 
 \textbf{Step 3.} Now we show the desired result.
 Since
 $V^{\pi^*,m^*}(t,\cdot,\cdot)\in W^{1,2}_p(D_N(t,0))\cap\Cc^{0,1}_\alpha(D_N(t,0))$, the It\^o--Krylov formula is applicable on $[t,(t+\varepsilon_0)\wedge\rho_{t,N}]$ and gives
 \begin{align}
  &V^{\pi^*,m^*}  (t,(t+\varepsilon_0)\wedge\rho_{t,N},  X_{(t+\varepsilon_0)\wedge\rho_{t,N}}^{\pi'})
  -  V^{\pi^*,m^*}(t,t,x)\notag\\
  &=  \int_t^{(t+\varepsilon_0)\wedge\rho_{t,N}}  \bigg[  \partial_sV^{\pi^*,m^*}(t,l,X_l^{\pi'})
  +  \tilde b(l,X_l^{\pi'},m_l^*,\pi'(l,X_l^{\pi'}))  \cdot D_xV^{\pi^*,m^*}(t,l,X_l^{\pi'})\notag\\
  &\qquad\qquad\quad\qquad  +\frac12\tr\left(  (\sigma\sigma^T)(l,X_l^{\pi'},m_l^*)  D_x^2V^{\pi^*,m^*}(t,l,X_l^{\pi'})  \right)
  \bigg]dl\notag\\
  &\quad+  \int_t^{(t+\varepsilon_0)\wedge\rho_{t,N}}  D_xV^{\pi^*,m^*}(t,l,X_l^{\pi'})^T  \sigma(l,X_l^{\pi'},m_l^*)dW_l. \label{eq:refined.verify.Ito}
 \end{align}
 The stopped transition law of $X^{\pi'}$ has an $L^q$ density on the bounded cylinder. Hence the a.e. strong PDE \eqref{eq:EHJB.linear} and the a.e. estimate \eqref{eq:refined.verify.HL} may be evaluated under the expectation in \eqref{eq:refined.verify.Ito}. Taking expectations and adding the running reward under $\pi'$ gives
 \begin{align}
  &\E_{t,x}\bigg[  V^{\pi^*,m^*}  (t,(t+\varepsilon_0)\wedge\rho_{t,N},  X_{(t+\varepsilon_0)\wedge\rho_{t,N}}^{\pi'})  -V^{\pi^*,m^*}(t,t,x)  
  +\int_t^{(t+\varepsilon_0)\wedge\rho_{t,N}}  \tilde r(l-t,X_l^{\pi'},m_l^*,\pi'(l,X_l^{\pi'}))dl
  \bigg]\notag\\
  &=  \E_{t,x}\int_t^{(t+\varepsilon_0)\wedge\rho_{t,N}}  \bigg\{  -\tilde b(l,X_l^{\pi'},m_l^*,\pi^*(l,X_l^{\pi'}))  \cdot D_xV^{\pi^*,m^*}(t,l,X_l^{\pi'})\notag\\
  &\qquad \qquad \qquad\qquad\qquad  -\tilde r(l-t,X_l^{\pi'},m_l^*,\pi^*(l,X_l^{\pi'}))  +  \tilde b(l,X_l^{\pi'},m_l^*,\pi'(l,X_l^{\pi'}))  \cdot D_xV^{\pi^*,m^*}(t,l,X_l^{\pi'})\notag\\
  &\qquad \qquad \qquad\qquad\qquad  +\tilde r(l-t,X_l^{\pi'},m_l^*,\pi'(l,X_l^{\pi'}))  \bigg\}dl. \label{eq:refined.verify.afterPDE}
 \end{align}
 By the definition of $H$, for every $(l,y)$ and every relaxed action $\pi'(l,y)$,
 $$
 \tilde b(l,y,m_l^*,\pi'(l,y))\cdot D_xV^{\pi^*,m^*}(t,l,y) +\tilde r(l-t,y,m_l^*,\pi'(l,y)) \leq H(t,l,y).
 $$
 The first two terms inside the braces in \eqref{eq:refined.verify.afterPDE} equal $-L(t,l,X_l^{\pi'})$. Therefore,
 \begin{align}
  \text{right-hand side of \eqref{eq:refined.verify.afterPDE}}
  &\leq  \E_{t,x}\int_t^{(t+\varepsilon_0)\wedge\rho_{t,N}}  [H(t,l,X_l^{\pi'})-L(t,l,X_l^{\pi'})]dl\notag\\
  &\leq  \int_t^{t+\varepsilon_0}2C|l-t|^{\alpha/2}dl  =  \frac{2C}{1+\alpha/2}\varepsilon_0^{1+\alpha/2}. \label{eq:refined.verify.finalbound}
 \end{align}
 The second inequality follows from \eqref{eq:refined.verify.HL}; as noted above, the exceptional zero-measure set is invisible under the transition density of $X^{\pi'}$.Combining \eqref{eq:refined.verify.beforeIto} and \eqref{eq:refined.verify.finalbound} yields
 $$
 J^{\pi'\otimes_{t,\varepsilon_0}\pi^*,m^*}(t,x) -J^{\pi^*,m^*}(t,x) \leq \frac{2C}{1+\alpha/2}\varepsilon_0^{1+\alpha/2} +\varepsilon.
 $$
Letting $\eps\downarrow0$, dividing by $\eps_0$, and then sending $\eps_0\downarrow0$ proves \eqref{eq:def.equipi}.
\end{proof}

\begin{Theorem}[Existence of relaxed mean-field equilibrium]\label{thm:equi.existence}
Let Assumptions \ref{assume.r} and \ref{assume.lipsa.U} hold and suppose $\nu\in\Pc_{2+\kappa}(\R^d)$. Then there exists a relaxed mean-field equilibrium $(\pi^*,m^*)$ satisfying Definition \ref{def:equi.relaxed}.
\end{Theorem}

\begin{proof}
Let $(\pi^*,m^*)$ be the fixed point obtained in Theorem \ref{thm:fixedpoint}. Proposition \ref{prop:verification} shows that $\pi^*$ satisfies the equilibrium response condition \eqref{eq:def.equipi}. The second identity in \eqref{eq:fixedpoint.conditions} gives
$$
m_t^*=\mu_t^{\pi^*,m^*}=\operatorname{law}(X_t^{\pi^*,m^*}),\qquad t\in[0,T],
$$
which is exactly the consistency condition \eqref{eq:def.equim}. Hence $(\pi^*,m^*)$ satisfies Definition \ref{def:equi.relaxed}.
\end{proof}

\begin{Remark}
The proof separates the two compactness mechanisms in the problem. The uniform $W^{1,2}_p$ and H\"older estimates yield stability of the auxiliary value function under weak convergence of relaxed policies, whereas the compact set $\mathscr{M}_\nu^{\kappa,C^*}$ controls the population flows. The Fokker--Planck duality argument is needed because the policy topology gives only weak-$*$ convergence of the controlled drift. Once these two stability results are available, the equilibrium Hamiltonian condition is obtained directly from the fixed-point property and no entropy regularization or vanishing-entropy limit is required.
\end{Remark}

\bibliographystyle{plain}
\bibliography{references}
\end{document}